\documentclass[11pt]{article}
\usepackage{graphicx}
\usepackage{subfigure}
\usepackage{booktabs} 

\usepackage{hyperref}

\usepackage[margin=1in]{geometry}                
\usepackage[english]{babel}
\usepackage[linesnumbered,ruled]{algorithm2e}

\usepackage{algpseudocode,braket}
\usepackage{amssymb,amsthm,amsmath,verbatim,array,setspace,bbm,algorithmicx}

\newtheorem{theorem}{Theorem}

\newtheorem{lemma}{Lemma}

\newtheorem{proposition}{Proposition}
\newtheorem{definition}[theorem]{Definition}
\newcommand*\samethanks[1][\value{footnote}]{\footnotemark[#1]}

\title{Principal Component Projection with Low-Degree Polynomials}

\author{Stephen D. Farnham \thanks{Department of Mathematics, Syracuse University, Syracuse, NY 13244, USA. Email: \texttt{sdfarnha@syr.edu} and \texttt{lshen03@syr.edu}} \and Lixin Shen \samethanks[1]  \and Bruce W. Suter \thanks{Air Force Research Laboratory, Rome, NY. Email: \texttt{bruce.suter@us.af.mil.}}}

\begin{document}
\maketitle

\begin{abstract}
In this paper, we consider approximations of principal component projection (PCP) without explicitly computing principal components. This problem has been studied in several recent works. The main feature of existing approaches is viewing  the PCP matrix as a matrix function. This underlying function is the composition of a step function with a rational function. To find an approximate PCP, the step function is approximated by a polynomial while the rational  function is evaluated by a fast ridge regression solver. In this work, we further improve this process by replacing the rational function with carefully constructed polynomials of low degree. We characterize the properties of polynomials that are suitable for approximating PCP, and establish an optimization problem to select the optimal one from those polynomials. We show theoretically and confirm numerically that the resulting approximate PCP approach with optimal polynomials is indeed effective for approximations of principal component projection.
\end{abstract}

\section{Introduction}
Principal component projection (PCP), an essential and useful tool in machine learning and statistics, is a mathematical procedure that projects high dimensional data onto a lower dimensional subspace. This subspace is defined by the principal components with the highest variance in the training data.  Mathematically, PCP can be achieved by  finding the top principal components of a matrix, through any principal component analysis (PCA) solver, and then projecting the underlying vector onto their span. Unfortunately, computing the principal components of a matrix is an expensive task, naturally raising the question of how one can efficiently project a vector onto the span of the top principal components of a matrix without performing PCA.

This question has recently been addressed in \cite{Allen-Zhu-Li:ICML:17,Frostig-Musco-Musco-Sidford:ICML:16} through an iterative algorithm based on black-box calls to a ridge regression routine. We briefly review the main idea behind this algorithm. Given a general matrix $A \in \mathbb{R}^{m \times n}$ and a vector $y \in \mathbb{R}^n$, we want to compute the projection of $y$ onto the span of the eigenvectors of $A^\top A$ corresponding to eigenvalues above a threshold $\lambda$. We denote this projection, i.e., the principal component projection of $y$, $P_{(A,\lambda)} y$. The key observation in \cite{Frostig-Musco-Musco-Sidford:ICML:16} is that the principal component projection operator $P_{(A,\lambda)}$ can be viewed as a matrix function in the following form
\begin{equation}\label{eq:key1}
P_{(A,\lambda)} = s_{\frac{1}{2}}(r_\lambda(A^\top A)),
\end{equation}
where
\begin{equation}\label{def:rlambda-shalf}
r_\lambda(x)=\frac{x}{x+\lambda} \quad \mbox{for $x\ge 0$} \quad \mbox{and} \quad
s_{\frac{1}{2}}(x)=\left\{
                  \begin{array}{ll}
                    0, & \hbox{$x <\frac{1}{2}$;} \\
                    1, & \hbox{$x \ge \frac{1}{2}$,}
                  \end{array}
                \right.
\end{equation}
are called the ridge regression function and the step function, respectively.  The observation~\eqref{eq:key1} provides a way for finding an approximation to $P_{(A,\lambda)} y$  through an approximation to $s_{\frac{1}{2}}(r_\lambda(A^\top A))y$. In \cite{Allen-Zhu-Li:ICML:17}, the step function $s_{\frac{1}{2}}$ is approximated by a degree $n$ polynomial, denoted $g_n$, which is obtained by approximating the signum function via a Chebyshev polynomial. As a result,  computing an approximation to $P_{(A,\lambda)} y$ amounts to evaluating $g_n(r_\lambda(A^\top A))y$.  Numerical experiments in \cite{Allen-Zhu-Li:ICML:17,Frostig-Musco-Musco-Sidford:ICML:16} demonstrated that $g_n(r_\lambda(A^\top A))y$ provides an approximation to $P_{(A,\lambda)} y$ for any vector $y$ in a stable and efficient manner.

Our goal in this paper is to develop a more efficient way of approximating the principal component projection than the approaches in \cite{Allen-Zhu-Li:ICML:17,Frostig-Musco-Musco-Sidford:ICML:16} in terms of computational efficiency. To this end, we take a closer look at the ridge regression function $r_\lambda$, which determines the computational cost of approximating the principal component projection $P_{(A,\lambda)} y$ via $g_n(r_\lambda(A^\top A))y$  under a given error tolerance. Firstly, the function $r_\lambda$ sends the eigenvalues of $A^\top A$, above and below the threshold $\lambda$,  to that of $r_\lambda(A^\top A)$, above and below $1/2$, respectively. Simultaneously,  all eigenvalues of $r_\lambda(A^\top A)$ lie in the interval $[0, 1]$. Secondly, the steepness of $r_\lambda$ around $\lambda$ determines the degree of the polynomial $g_n$ (a detailed description about this statement will be given in Section 2). Roughly speaking, the greater the steepness of $r_\lambda$ around $\lambda$ is, the lesser the degree of $g_n$ is required to be. Lastly, once the polynomial $g_n$ is chosen, the basic operation in evaluating $g_n(r_\lambda(A^\top A))y$ is the computation of $r_\lambda(A^\top A)v$ for any vector $v$. Since $r_\lambda(A^\top A)v = (A^\top A+\lambda I)^{-1} A^\top A v$, it can be approximated efficiently by using any fast ridge regression solver. With the above discussion, and to have a more efficient way of approximating PCP along the line in \cite{Allen-Zhu-Li:ICML:17,Frostig-Musco-Musco-Sidford:ICML:16}, one should consider replacing $r_\lambda$ by an alternative function, say $r$, which not only possesses the above properties, but also is steeper than $r_\lambda$ around $\lambda$, and is simply enough for $r(A^\top A) v$ to be computed explicitly and accurately with fewer computations than $r_\lambda(A^\top A)v$.  In this paper, we show theoretically that a polynomial with low degree can serve our goal very well.

Our main contributions are highlighted as follows:
\begin{itemize}
\item We introduce the concepts of $(\lambda,\gamma)$-admissible polynomials and a $(\lambda,\gamma)$-admissible gap associated with these polynomials, where $\lambda \in (0,1)$, $\gamma \in (0,1)$, and $\lambda(1+\gamma)<1$. A $(\lambda,\gamma)$-admissible polynomial will map $[0, (1-\gamma)\lambda]$ to $[-1, 0)$ and $[(1+\gamma)\lambda, 1]$ to $(0, 1]$. The optimal $(\lambda,\gamma)$-admissible polynomial in terms of its steepness around $\lambda$ is defined.

\item We show that an optimal $(\lambda,\gamma)$-admissible polynomial can be derived from an optimal $(1-\lambda,\gamma')$-admissible polynomial with a proper choice of $\gamma'$. We  present the explicit forms of the optimal $(\lambda,\gamma)$-admissible polynomials with degree $1$ and $2$ through an optimization procedure.

\item We show theoretically and confirm numerically which one, among the ridge regression function and  the optimal $(\lambda,\gamma)$-admissible polynomials with degree $1$ and $2$, should be chosen in approximating PCP of vectors for a given threshold $\lambda$.
\end{itemize}

This paper is organized as follows. In the next section, we make necessary preparations by recalling the concept of $(\gamma, \epsilon)$-approximate PCP and recalling the \texttt{QuickPCP} algorithm for  $(\gamma, \epsilon)$-approximate PCP. This review motivates us to propose a $(\gamma, \epsilon)$-approximate PCP algorithm using low-degree polynomials. In Section~\ref{sec:Admissible-Ploy}, we introduce a class of admissible polynomials for PCP. We pose a general optimization problem for finding the optimal admissible polynomial in the sense of maximizing the admissible gap of admissible polynomials. In particular, the optimal admissible polynomial of degree 1 is explicitly given. Section~\ref{sec:degree-2} gives a complete study about the optimal admissible polynomial of degree 2. The concrete form of the  optimal admissible polynomial depends on the parameters $\gamma$ and $\lambda$. Section~\ref{sec:RR-Poly} is devoted to answering the question of under what scenarios the optimal $(\lambda,\gamma)$-admissible polynomials with degree $1$ and $2$ will outperform  the ridge regression function when they are used to compute $(\gamma, \epsilon)$-approximate PCP. In Section~\ref{sec:experiments}, we confirm numerically the theoretical results given in the previous sections. We draw our conclusions in Section~\ref{sec:conclusion}.


\section{Preliminaries}\label{sec:prel}
This section will provide notation used throughout the paper along with definitions and some key concepts.

We denote by $\mathbbm{1}: E \rightarrow \{0,1\}$ the indicator function for event $E$, that if the event $E$ holds, then $\mathbbm{1}(E)$ equals $1$; otherwise $0$.  We denote by $\|v\|$ the Euclidean norm of a vector $v$, and by $\|A\|$ the spectral norm of a matrix $A$.

For any matrix $A$ of size $m \times d$ with rank $r$, the singular value decomposition (SVD) of $A$ has a form of $A=U\Sigma V^\top$, where $U\in \mathbb{R}^{m \times r}$ and $V \in \mathbb{R}^{n \times r}$ both have orthonormal columns and $\Sigma \in \mathbb{R}^{r \times r}$ is a diagonal matrix. Moreover, $\Sigma = \mathrm{diag}(\sigma_1, \sigma_2, \ldots, \sigma_r)$, where $\sigma_1 \ge \sigma_2 \ge \cdots \ge \sigma_r \ge 0$ are the singular values of $A$ in decreasing order. For this matrix $A$ and a positive number $\lambda \in (0, \sigma_1^2]$, the projection matrix associated with $A$ at threshold $\lambda$, denoted by $P_{(A,\lambda)}$, is given as follows
$$
P_{(A,\lambda)} = \sum_{j\in \{i: \sigma_i^2 \ge \lambda\}} v_j v^\top_j,
$$
which is another way of viewing \eqref{eq:key1}.  Correspondingly, for any vector $\chi \in \mathbb{R}^d$, we call the vector $P_{(A,\lambda)} \chi$ the principal component projection (PCP) of $\chi$ at threshold $\lambda$.

Let us recall the definition of matrix functions. Let $f$ be a real function defined on $\mathbb{R}$. If $D=\mathrm{diag}(d_1, d_2, \ldots, d_r)$ is a diagonal matrix with diagonal entries $d_i$, we denote $f(D)=\mathrm{diag}(f(d_1), f(d_2), \ldots, f(d_r))$. If $B$ is a real symmetric matrix, we define $f(B) = U f(\Sigma) U^\top$, where $B=U\Sigma U^\top$ is the singular value decomposition (SVD) of $B$.

\subsection{The $(\gamma, \epsilon)$-approximate PCP}

Our goal is to efficiently compute the principal component projection of a vector. Intuitively, for any given matrix $A$, threshold $\lambda$, and vector $\chi$, the vector $P_{(A,\lambda)} \chi$ can be obtained by first computing the SVD of $A$, then choosing the right singular vectors (i.e., principal components) whose corresponding singular values are greater than $\sqrt{\lambda}$ followed by projecting the vector $\chi$ onto the subspace spanned by those selected singular vectors. Unfortunately, the previously outlined approach is usually computationally intractable, especially for large scale matrices. Therefore, one resorts to finding an approximation to $P_{(A,\lambda)} \chi$ efficiently within a given tolerance. Mathematically, such a feasible approximation is characterized in the following definition introduced in \cite{Frostig-Musco-Musco-Sidford:ICML:16}.
\begin{definition}\label{def:app-PCP}
For a given matrix $A$ of size $m \times d$ and a threshold $\lambda$, we say $\mathcal{P}_{(A,\lambda)} (\chi)$ is a $(\gamma, \epsilon)$-approximate PCP for $\chi \in \mathbb{R}^d$ at threshold $\lambda$ if
\begin{enumerate}
  \item $\|P_{(A,(1+\gamma)\lambda)} (\mathcal{P}_{(A,\lambda)} (\chi)-\chi) \| \le \epsilon \|\chi\|$;
  \item $\|(I-P_{(A,(1-\gamma)\lambda)}) \mathcal{P}_{(A,\lambda)} (\chi)\| \le \epsilon \|\chi\|$;
  \item $|\langle v_i,  \mathcal{P}_{(A,\lambda)} (\chi) -\chi\rangle| \le |\langle v_i, \chi\rangle|+\epsilon \|\chi\|$ hold for all $i$ such that $\sigma_i^2 \in [(1-\gamma)\lambda, (1+\gamma)\lambda]$. Here, $\sigma_i^2$ is the $i$-th singular value of $A^\top A$ and $v_i$ is the corresponding singular eigenvector.
\end{enumerate}
\end{definition}
In the above definition, the first condition implies that the projections of $\mathcal{P}_{(A,\lambda)} (\chi)$ and $\chi$ onto the subspace spanned the principal components with corresponding singular values above the threshold $(1+\gamma)\lambda$ are almost identical; the second condition says that the projection of $\mathcal{P}_{(A,\lambda)} (\chi)$ onto the subspace spanned by the principal components corresponding to singular values below the threshold $(1-\gamma)\lambda$ is close to zero; and the last condition indicates that for each singular value in $[\sqrt{(1-\gamma)\lambda}, \sqrt{(1+\gamma)\lambda}]$, $\langle v_i,  \mathcal{P}_{(A,\lambda)} (\chi)\rangle$ (the projection of $\mathcal{P}_{(A,\lambda)} (\chi)$ onto $v_i$) is between $0$ and $\langle v_i, \chi\rangle$ up to an error $\epsilon \|\chi\|$.


\subsection{The reformulation of the projection matrix $P_{(A,\lambda)}$}
The central question is how to quickly construct $\mathcal{P}_{(A,\lambda)}$ that can yield a $(\gamma, \epsilon)$-approximate PCP for $\chi \in \mathbb{R}^d$ at threshold $\lambda$. This relies on a further understanding of the structure of $P_{(A,\lambda)}$ the projection matrix associated with $A$ at threshold $\lambda$. Using the notion of matrix functions, a key idea introduced in \cite{Frostig-Musco-Musco-Sidford:ICML:16} is to rewrite the projection matrix $P_{(A,\lambda)}$ as follows
\begin{equation*}\label{eq:s-shalf}
P_{(A,\lambda)} = s_{\frac{1}{2}}(r_\lambda(A^\top A)),
\end{equation*}
where both $r_\lambda$ and $s_{\frac{1}{2}}$ are given in \eqref{def:rlambda-shalf}. Clearly, the above equation reveals that $P_{(A,\lambda)}$ can be achieved via two steps: computing $r_\lambda(A^\top A)$ and then $s_{\frac{1}{2}}(r_\lambda(A^\top A))$. Obviously,
$$
r_\lambda(A^\top A) = (A^\top A + \lambda I)^{-1} A^\top A,
$$
which is a smooth approximation to $P_{(A,\lambda)}$ in the sense that both $P_{(A,\lambda)}$  and $r_\lambda(A^\top A)$ have identical principal components and $r_\lambda(\sigma_i^2)$, the eigenvalue of $r_\lambda(A^\top A)$, approaches to $1$ when $\sigma_i^2$, the eigenvalue of $A^\top A$, is much larger than $\lambda$ and $0$, when it is much smaller than $\lambda$. The matrix $s_{\frac{1}{2}}(r_\lambda(A^\top A))$ can be viewed as a sharpened version of $r_\lambda(A^\top A)$ via the step function $s_{\frac{1}{2}}$.

With this decomposition of $P_{(A,\lambda)}$, the previously posed question becomes how to efficiently evaluate  $r_\lambda(A^\top A) \chi$ and $s_{\frac{1}{2}}(r_\lambda(A^\top A))\chi$ for any vector $\chi$. Since the vector $r_\lambda(A^\top A)\chi$ is the unique solution of the convex ridge regression problem
\begin{equation}\label{eq:ridge}
\min \{\|Ax-b\|^2+\lambda \|x\|^2: x \in \mathbb{R}^d\}
\end{equation}
with $b=A \chi$,  many fast solvers are available for solving ridge regression problem \eqref{eq:ridge},  see, for example, \cite{Johnson-Zhang:NISP:2013,Nesterov:83,Shai-Zhang:JMLR:2013}. In particular, a vector denoted by $\mathcal{R}(A, \lambda, A^\top A \chi)$,  produced by a ridge regression algorithm for solving \eqref{eq:ridge}, is an $\epsilon$-approximation to $r_\lambda(A^\top A)\chi$ if for any vector $u\in \mathbb{R}^d$,
$$
\|\mathcal{R}(A, \lambda, u)-(A^\top A + \lambda I)^{-1} u\| \le \epsilon \|u\|
$$
with running time of $O(\mathrm{nnz}(A) \lambda^{-1/2} \log(1/\epsilon)$. Here, $\mathrm{nnz}(A)$ is the number of nonzero entries of $A$.

Next, we review how the vector $s_{\frac{1}{2}}(r_\lambda(A^\top A))\chi$ is approximated in \cite{Allen-Zhu-Li:ICML:17,Frostig-Musco-Musco-Sidford:ICML:16}. Note that
$$
s_{\frac{1}{2}}(r_\lambda(A^\top A))= \frac{1}{2}(I+\mathrm{sgn}(2r_\lambda(A^\top A)-I)),
$$
where $\mathrm{sgn}$ is the signum function with
$$
\mathrm{sgn}(\xi)=\left\{
                    \begin{array}{ll}
                      1, & \hbox{if $\xi \ge 0$;} \\
                      -1, & \hbox{if $\xi<0$.}
                    \end{array}
                  \right.
$$

A polynomial approximation to the matrix sign function in \cite{Allen-Zhu-Li:ICML:17} was obtained through constructing a polynomial $q(x)$ that approximates $(\frac{1+\kappa-x}{2})^{-\frac{1}{2}}$ for some positive number $\kappa$ and then setting $g(x) = x q(1+\kappa-2x^2)$ as an approximation to the signum function $\mathrm{sgn}(x)$. In fact, the polynomial $q$ in \cite{Allen-Zhu-Li:ICML:17} is chosen to be the Chebyshev interpolation polynomial. The degree-$n$ Chebyshev interpolation polynomial for the function  $(\frac{1+\kappa-x}{2})^{-\frac{1}{2}}$  is
\begin{equation*}\label{eq:chenyshev}
q_n(x):=\sum_{k=0}^n c_k \mathcal{T}_k(x),
\end{equation*}
where
$$
c_k = \frac{2-\mathbbm{1}[k=0]}{n+1}\sum_{j=0}^n \sqrt{2} \cos \left(\frac{k(j+1/2)\pi}{n+1}\right) \left(1 + \kappa - \cos \frac{(j+1/2)\pi}{n+1} \right )^{-1/2}.
$$
The Chebyshev polynomials of the first kind $\{\mathcal{T}_k(x)\}_{k\ge 0}$ are defined by the recurrence relation:
$\mathcal{T}_0(x) = 1$, $\mathcal{T}_1(x)= x$, and $\mathcal{T}_{n+1}(x) = 2x \cdot \mathcal{T}_n(x) - \mathcal{T}_{n-1}(x)$.
With the Chebyshev interpolation polynomial, we define
\begin{equation}\label{eq:chenyshev-sign}
g_n(x) := x  q_n(1+\kappa-2x^2).
\end{equation}
The following result in \cite{Allen-Zhu-Li:ICML:17} shows that the polynomial $g_n$ indeed approximates the signum function.
\begin{lemma}[Theorem 5.1 in \cite{Allen-Zhu-Li:ICML:17}]\label{lemma:Allen-Zhu-Chebyshev}
For every $\alpha \in (0,1]$, $\epsilon \in (0,1/2)$, choosing $\kappa=2\alpha^2$, as long as $n \ge \frac{1}{\sqrt{2}\alpha}\log\frac{3}{\epsilon \alpha^2}$, then
\begin{itemize}
  \item[(i)] $|g_n(x)-\mathrm{sgn}(x)|\le \epsilon$ for all $x \in [-1,-\alpha]\cup [\alpha, 1]$;
  \item[(ii)]  $0\le g_n(x)\le 1$ for all $x\in [0,\alpha]$ and $-1\le g_n(x)\le 0$ for all $x\in [-\alpha,0]$.
\end{itemize}
\end{lemma}

\subsection{\texttt{QuickPCP}: An approach for $(\gamma, \epsilon)$-approximate PCP}

With the discussion given in the previous subsection, computing a $(\gamma, \epsilon)$-approximate PCP for $P_{(A,\lambda)}\chi$ becomes evaluating $g_n(2r_\lambda(A^\top A)-I) \chi$, where $g_n$ is given in \eqref{eq:chenyshev-sign}. As it is well known in numerical analysis,  the Clenshaw algorithm \cite{Clenshaw:MC:1955}, as a generalization of Horner's method for evaluating a linear combination of monomials, is often used to evaluate a linear combination of Chebyshev polynomials. The work in \cite{Allen-Zhu-Li:ICML:17} adopted the Clenshaw algorithm and a ridge regression solver to compute $g_n(A^\top A) \chi$, resulting in a so-called \texttt{QuickPCP} in Algorithm~\ref{alg:allen-zhu}.
\begin{algorithm} \label{alg:allen-zhu}
\SetKwInOut{Input}{Input}
\SetKwInOut{Output}{Output}
\caption{\texttt{QuickPCP}$(A,\chi,\lambda,\gamma,n)$}
	
	\Input{ $A \in \mathbb{R}^{m \times d}, \text{ } \|A\| =1$; \quad $ \chi \in \mathbb{R}^d$;\quad  $\lambda > 0, $ threshold; \newline
	 \quad $\gamma \in (0,1),$ PCP approximation ratio \quad $n$, degree of Chebyshev approximation}
	
	\Output{vector that approximates $\zeta \approx P_{(A,\lambda)}(\chi)$}
	
	$\gamma \leftarrow \max\{\gamma, \frac{\log(n)}{n}\}$
	
	$\alpha = \frac{ \gamma}{2+\gamma}$
	
	$\kappa = 2 \alpha^2$
		
	Define $c_k = \frac{2-\mathbbm{1}[k=0]}{n+1} \sum_{j=0}^n \sqrt{2} \cos\left(\frac{k(j+.5)\pi}{n+1}\right)\left(1+\kappa - \cos\left(\frac{(j+.5)\pi}{n+1}\right)\right)^{-1/2} $\
	
	$b_{n+1} = 0$, $b_n =c_n  \chi$\
	
	\For{$ r = n-1$ \textbf{to} $0$}{
	
	$ v =  \mathcal{R}(A,\lambda, (A^\top A-\lambda I)b_{r+1})$
	
	$w=(1+\kappa)b_{r+1}- 2 \mathcal{R}(A,\lambda,(A^\top A-\lambda I)v)$ \
	
	$b_r = 2w-b_{r+2}+c_r \chi$\
	
}
$u=\mathcal{R}(A,\lambda, (A^\top A-\lambda I)b_0-w)$\

\Return{$\frac{1}{2} (u+\chi)$}
	
\end{algorithm}

Lines 6-10 of Algorithm~\ref{alg:allen-zhu} are the realization of Clenshaw algorithm for matrices. There are a total of $2n+1$ calls for a ridge regression solver $\mathcal{R}$. The convergence of \texttt{QuickPCP} is summarized in the following result.
\begin{lemma}[Theorem 7.3 of \cite{Allen-Zhu-Li:ICML:17}] \label{Theorem7.3}
 Given $A \in \mathbb{R}^{m\times d}$ and $\lambda, \gamma \in (0,1)$, assume that the  singular values of $A$ are in the range $[0,\sqrt{(1-\gamma)\lambda}] \cup [\sqrt{(1+\gamma)\lambda},1]$. Given $\chi \in \mathbb{R}^d$, let $\zeta^*=P_{(A,\lambda)}\chi$ be the exact PCP solution. If $\mathcal{R}$ is an $\epsilon'$-approximate ridge regression solver, then the output $\zeta=\texttt{QuickPCP}(A,\chi,\lambda,\gamma ,n)$ satisfies
\begin{equation*}
\|\zeta^*-\zeta\|\leq \epsilon \|\chi\|,
\end{equation*}
as long as $n \geq  \frac{2+\gamma}{\sqrt{2} \gamma} \log\left( \frac{3(2+\gamma)^2}{\epsilon \gamma^2}\right)$ and $\log(\frac{1}{\epsilon'}) = \Theta (\log(\frac{n}{\epsilon \gamma}))$.
\end{lemma}


We can observe that the computational cost for computing approximate PCP, through \texttt{QuickPCP}, depends on both the degree of the polynomial $g_n$ and the cost of ridge regression. According to Lemma~\ref{lemma:Allen-Zhu-Chebyshev}, the degree of the polynomial $g_n$ is determined by the parameter $\alpha$ under the same accuracy $\epsilon$. An increased value of $\alpha$ yields a lower degree polynomial, therefore requiring less computation time for evaluation. The value of $\alpha$ reflects the steepness of the function $r_\lambda$ at $\lambda$, i.e., the steepness of $2r_\lambda-1$ at $\lambda$. This value is
\begin{equation}\label{eq:alpha-rr}
\alpha=\min \{1-2 r_\lambda((1-\gamma)\lambda), 2 r_\lambda((1+\gamma)\lambda)-1\}=\frac{\gamma}{2+\gamma}.
\end{equation}
The cost of ridge regression varies and depends on several factors, including the distribution of the eigenvalues of $A^\top A$ away from $\lambda$. Since each call to ridge regression attempts to approximate $(A^\top A + \lambda I)^{-1}(A^\top A - \lambda I) \chi$, we can view this routine as being at least as complex as evaluating a degree 2 matrix polynomial.

With all of this, the question is to determine whether or not the calls to ridge regression can be replaced with a different function that will reduce the convergence time of \texttt{QuickPCP}. In the rest of this paper, we show theoretically that a polynomial with low degree can serve our goal very well.

\section{Admissible Polynomials for PCP} \label{sec:Admissible-Ploy}
As discussed in the preceding section, we want to replace the rational function $2r_\lambda-1$ with a polynomial which can reduce the convergence time of \texttt{QuickPCP}. In this section, the requirements for such polynomials are defined and an optimal one among these polynomials with a fixed degree under certain sense is posted as a solution to an optimization problem. To facilitate the discussion and characterize these polynomials, we introduce two concepts, namely, $(\lambda, \gamma)$-admissible polynomials and the optimal $(\lambda, \gamma)$-admissible polynomials. We show that a $(\lambda, \gamma)$-admissible polynomial can be derived from a $(1-\lambda, \gamma')$-admissible polynomial of the same degree with a proper $\gamma'$. This result allows us to focus on the construction of $(\lambda, \gamma)$-admissible polynomials with $\lambda \le 1/2$.  At the end of the section, we give the explicit expression of the optimal $(\lambda, \gamma)$-admissible polynomial of degree $1$.

We begin with pointing out the properties of $2r_\lambda-1$ that allow it to be effective in producing approximate PCP. The first property is that the function maps the interval $[0,1]$ inside the interval $[-1,1]$. This allows for the composition of ridge regression with a Chebyshev polynomial. Next, the interval $[0, \lambda-\lambda\gamma)$ is mapped to the interval $[-1,2r_\lambda(\lambda-\lambda\gamma)-1) \subset [-1,0)$ and the interval $(\lambda+\lambda\gamma,1]$ is mapped to the interval $(2r_\lambda(\lambda+\lambda\gamma)-1,1]\subset (0,1]$. Then, by choosing $\alpha$ given in \eqref{eq:alpha-rr},  we can invoke Lemma~\ref{lemma:Allen-Zhu-Chebyshev} to determine the required degree of the Chebyshev polynomial  to ensure a desirable approximate PCP. These properties motivate the following definition.

\begin{definition}\label{def:admissible-poly}
For $\lambda \in (0,1)$ and $\gamma \in (0, 1)$  with $\lambda(1+\gamma)<1$,  a polynomial $p$ is said to be $(\lambda, \gamma)$-admissible if the following conditions are satisfied:
\begin{itemize}
\item[(i)] $p(\lambda)=0$;
\item[(ii)] for all $x\in [0, (1-\gamma)\lambda]$, $-1 \le p(x) \le p((1-\gamma)\lambda) <0$; and
\item[(iii)] for all $x\in [(1+\gamma)\lambda,1]$, $0 < p((1+\gamma)\lambda)\le p(x) \le 1$.
\end{itemize}
Furthermore, associated with this $(\lambda, \gamma)$-admissible polynomial $p$, the positive value, denoted by $\alpha(p;\lambda, \gamma)$,
$$
\alpha(p;\lambda, \gamma) = \min\{-p((1-\gamma)\lambda), p((1+\gamma)\lambda)\}
$$
is called the $(\lambda, \gamma)$-admissible gap of $p$.
\end{definition}

The first part of the definition states that the polynomial should send $\lambda$ to $0$, which also occurs with $2r_\lambda -1$. The second and third parts force the range of the polynomial to be between $-1$ and $1$ on the interval $[0,1]$ so that it may be composed with a Chebyshev polynomial. These conditions also bound the output of the polynomial away from $0$ when it is evaluated outside of the interval $(\lambda(1-\gamma), \lambda(1+\gamma))$. This was inspired by the result of Lemma \ref{lemma:Allen-Zhu-Chebyshev}.

With this definition, a $(\lambda, \gamma)$-admissible polynomial $p$ can be used in place of $2r_\lambda -1$. By replacing $\alpha=\frac{\gamma}{2+\gamma}$ (see \eqref{eq:alpha-rr}), used in ridge regression approach, with $\alpha(p;\lambda,\gamma)$, and adjusting the degree of the Chebyshev polynomial accordingly, the hypotheses for the main result (Lemma~\ref{Theorem7.3}) of \cite{Allen-Zhu-Li:ICML:17} will still hold.

An immediate result follows from the definition of $(\lambda,\gamma)$-admissible polynomials. This result simplifies their constructions and will be used throughout the rest of the paper.
\begin{proposition}\label{prop:reflection-lambda}
A polynomial $p$ is $(\lambda, \gamma)$-admissible if and only if the polynomial $-p(1-\cdot)$ is $(\lambda', \gamma')$-admissible, where $\lambda'=1-\lambda$ and $\gamma'=\frac{\lambda \gamma}{1-\lambda}$. The $(\lambda, \gamma)$-admissible gap of $p$ is identical to the $(\lambda', \gamma')$-admissible gap of $-p(1-\cdot)$.
\end{proposition}
\begin{proof} \ \ 
Assume that   $p$ is $(\lambda, \gamma)$-admissible. Write $q=-p(1-\cdot)$. Immediately, $q(\lambda')=-p(1-\lambda')=-p(\lambda)=0$. For $x\in [0, (1-\gamma')\lambda']$, then $1-x \in [\lambda(1+\gamma),1]$, therefore, $0<p((1+\gamma)\lambda) \le p(1-x) \le 1$ which implies $-1\le q(x) \le q((1-\gamma')\lambda')<0$; for $x\in [(1+\gamma')\lambda',1]$, then $1-x \in [0,\lambda(1-\gamma)]$, therefore $-1 \le p(1-x) \le p((1-\gamma)\lambda) <0$ which implies $0 < q((1-\gamma')\lambda') \le q(x) \le 1$. Hence, $q$ is $(\lambda', \gamma')$-admissible.

The reverse part of this proposition can be shown similarly. Furthermore, since $p((1-\gamma)\lambda)=-q((1+\gamma')\lambda')$ and   $p((1+\gamma)\lambda)=-q((1-\gamma')\lambda')$, we know that $\alpha(p;\lambda, \gamma)=\alpha(q;\lambda', \gamma')$. This completes the proof.
\end{proof}

The above proposition indicates that we only need to study the case of $\lambda \le 1/2$ for admissible polynomials since any admissible polynomial for $\lambda>1/2$ can be derived directly from an admissible polynomial with $\lambda<1/2$.

Our ultimate goal is to replace the rational function $2r_\lambda-1$ with a suitable $(\lambda,\gamma)$-admissible polynomial, $p$. There are two general requirements for this polynomial $p$. One is that the computational cost for the evaluation of $p(A^\top A) y$ for any vector $y$ should be lower than approximating $r_\lambda(A^\top A)y$, and the other is that the degree of Chebyshev polynomial, which can approximate the signum function well on the intervals away from $[-\alpha(p;\lambda, \gamma), \alpha(p;\lambda, \gamma)]$, should be low as well. Apparently, without considering the specific structures of $A^\top A$, the computational cost for evaluation of $p(A^\top A) y$ for any vector $y$ is determined by the degree of $p$ and the sparsity of $A^\top A$. Therefore, among all $(\lambda,\gamma)$-admissible polynomials with fixed degree $n$, the one yielding the lowest degree Chebyshev polynomial by Lemma~\ref{lemma:Allen-Zhu-Chebyshev} is the solution of the following optimization problem:
\begin{equation}\label{model:general}
p_{n}^{\lambda, \gamma}= \mathrm{arg}\max \{\alpha(p;\lambda, \gamma): \mbox{$p$ is any $(\lambda,\gamma)$-admissible polynomial with degree $n$} \}.
\end{equation}
We denote by $\alpha_n(\lambda, \gamma):=\alpha(p_{n}^{\lambda, \gamma};\lambda, \gamma)$ the optimal gap of $(\lambda,\gamma)$-admissible degree-$n$  polynomials.

By Proposition~\ref{prop:reflection-lambda}, we immediately have
$$
\alpha_n(\lambda, \gamma) =  \alpha_n\left(1-\lambda, \frac{\lambda \gamma}{1-\lambda}\right).
$$
To end this section, the explicit form of $p_{1}^{\lambda, \gamma}$ and  the value $\alpha_1(\lambda, \gamma)$ are given for $\lambda \le 1/2$.
\begin{proposition}\label{prop:n=1}
For $\lambda \in (0, \frac{1}{2}]$  and $\gamma \in (0,1)$, we have
\begin{equation}\label{eq:optimal-n=1}
p_{1}^{\lambda, \gamma}(x) = \frac{x-\lambda}{1-\lambda} \quad \mbox{and} \quad
 \alpha_1(\lambda, \gamma)  = \frac{\lambda \gamma}{1-\lambda}.
\end{equation}
\end{proposition}
\begin{proof}\ \ By the first condition in Definition~\ref{def:admissible-poly}, a $(\lambda,\gamma)$-admissible polynomial $p$ of degree $1$ must be in a form of $p(x)=a(x-\lambda)$. The second condition $p(0) \ge -1$ implies $a \le 1/\lambda$ while the third condition $p(1) \le 1$ implies $a \le 1/(1-\lambda)$. Hence $0<a \le 1/(1-\lambda)$ because $\lambda \le 1/2$. A direct computation gives $\alpha(p;\lambda, \gamma) = a \lambda \gamma$ whose maximal value will be achieved with $a=1/(1-\lambda)$. Hence, both $p_{1}^{\lambda, \gamma}$ and $\alpha_1(\lambda, \gamma)$ given in \eqref{eq:optimal-n=1} hold.
\end{proof}

The next section will be devoted to find $p_{2}^{\lambda, \gamma}$ and  $\alpha_2(\lambda, \gamma)$ with $\lambda < 1/2$.

\section{Optimal $(\lambda, \gamma)$-Admissible Polynomial of Degree $2$} \label{sec:degree-2}
The construction of the optimal  $(\lambda, \gamma)$-admissible polynomial of degree 2 is more complicated than its degree 1 counterpart.  Since all $(\lambda, \gamma)$-admissible polynomials vanish at $\lambda$, these polynomials with degree 2 will take a form of
\begin{equation}\label{eq:deg2}
p(x)=(x-\lambda)(ax-c),
\end{equation}
where $a$ and $c$ are two parameters depending on the requirements of admissible conditions. This quadratic function $p$ is convex if $a>0$ and concave if $a< 0$. We note that if $a=0$, then $c$ must be negative, the resulting polynomial is called the degenerated $(\lambda, \gamma)$-admissible polynomial, i.e., $(\lambda, \gamma)$-admissible degree $1$ polynomial.

For a polynomial given in \eqref{eq:deg2}, we observe that
$$
p(\lambda(1-\gamma)) + p(\lambda(1+\gamma)) = 2a \lambda^2 \gamma^2.
$$
Thus, if $p$ given in \eqref{eq:deg2} is $(\lambda, \gamma)$-admissible, then its admissible gap is
\begin{equation}\label{eq:deg2-gap}
\alpha(p; \lambda, \gamma)=\begin{cases}
 p(\lambda +\lambda\gamma) & \mbox{if $a \le 0$;}\\
-p(\lambda-\lambda\gamma) & \mbox{if $a>0$}.
 \end{cases}
\end{equation}
For the purpose of efficiently approximating PCP, the following result shows that the admissible gap for any convex quadratic polynomial cannot exceed the optimal gap of $(\lambda, \gamma)$-admissible degree 1 polynomial.
\begin{lemma}\label{lemma:a>0}
For any $(\lambda, \gamma)$-admissible polynomial $p$ given in \eqref{eq:deg2}, if $a>0$ and $\lambda \le 1/2$, then the $(\lambda, \gamma)$-admissible gap of $p$ is always less than the optimal gap of $(\lambda, \gamma)$-admissible polynomials with degree $1$, that is
$$
\alpha(p; \lambda, \gamma) \le \alpha_1(\lambda, \gamma).
$$
\end{lemma}
\begin{proof} \ \ From Proposition~\ref{prop:n=1} and \eqref{eq:deg2-gap}, we have that $\alpha_1(\lambda,\gamma)=\frac{\lambda\gamma}{1-\lambda}$ and $\alpha(p; \lambda, \gamma)=-\lambda\gamma(a\lambda(1-\gamma)-c)$. Then
$$
\alpha_1(\lambda, \gamma)-\alpha(p; \lambda, \gamma)=\frac{\lambda\gamma}{1-\lambda}(1-(1-\lambda)(a\lambda(1-\gamma)-c)).
$$
Note that $1-(1-\lambda)(a\lambda(1-\gamma)-c)=1-p(1)+a(1-\lambda)(1-\lambda(1-\gamma))$ which is positive. This completes the proof.
\end{proof}

Therefore, from Lemma~\ref{lemma:a>0} we should investigate $(\lambda, \gamma)$-admissible concave quadratic polynomials in the form \eqref{eq:deg2} for $\lambda \in (0, 1/2]$. In particular, for $\lambda=1/2$, we have the following result.
\begin{lemma}\label{lemma:a<0andlambda=1/2}
For any $(1/2, \gamma)$-admissible polynomial $p$ given in \eqref{eq:deg2}, if $a<0$, then the $(1/2, \gamma)$-admissible gap of $p$ is always less than the optimal gap of $(1/2, \gamma)$-admissible degree 1 polynomial, that is
$$
\alpha(p; 1/2, \gamma) \le \alpha_1(1/2, \gamma).
$$
\end{lemma}
\begin{proof} \ \ Define $q=-p(1-\cdot)$. By Proposition~\ref{prop:reflection-lambda}, $q$ is also a $(1/2, \gamma)$-admissible polynomial with $\alpha(q;1/2,\gamma)=\alpha(p;1/2,\gamma)$. Note that $q(x)=(x-1/2)(-ax-(c-a))$. By Lemma~\ref{lemma:a>0}, $\alpha(q;1/2,\gamma)\le \alpha_1(1/2,\gamma)$. Therefore, $\alpha(p; 1/2, \gamma) \le \alpha_1(1/2, \gamma)$.
\end{proof}

The implication from both Lemma~\ref{lemma:a>0} and Lemma~\ref{lemma:a<0andlambda=1/2} is that we should pay attention on the construction of an optimal $(\lambda, \gamma)$-admissible concave quadratic polynomial in \eqref{eq:deg2} with $\lambda \in (0,1/2)$. In the case of $a<0$, equation~\eqref{eq:deg2-gap} states that
\begin{equation}\label{eq:gap-a<0}
\alpha(p; \lambda, \gamma)= \lambda^2\gamma(1+\gamma)a-\lambda \gamma c.
\end{equation}

We now show that the gap $\alpha(p; \lambda, \gamma)$ in \eqref{eq:gap-a<0} for any $(\lambda, \gamma)$-admissible concave quadratic function is always greater than the product of the optimal gap of $(\lambda, \gamma)$-admissible degree-$1$ polynomial and the maximal value of the quadratic function over the interval $[\lambda, 1]$.
\begin{lemma}\label{lemma:a<0}
For any $(\lambda, \gamma)$-admissible concave quadratic polynomial $p$ given in \eqref{eq:deg2} with $\lambda \in(0, 1/2)$, we have
$$
\alpha(p;\lambda,\gamma) \ge \alpha_1(\lambda,\gamma) \cdot \max_{x\in [\lambda, 1]} p(x).
$$
\end{lemma}
\begin{proof}\ \ We know that $\alpha_1(\lambda,\gamma)= \frac{\lambda\gamma}{1-\lambda}$ from \eqref{eq:optimal-n=1} and $\alpha(p; \lambda, \gamma)= \lambda\gamma(\lambda(1+\gamma)a-c)$ from \eqref{eq:gap-a<0}.  It suffices to show that
$$
\lambda(1+\gamma)a-c \ge \frac{1}{1-\lambda}\cdot \max_{x\in [\lambda, 1]} p(x).
$$
Let $x^*$ be a point in $[\lambda, 1]$ such that $p(x^*)=\max_{x\in [\lambda, 1]} p(x)$. By the third requirement in the definition of admissible polynomial, we know that $\lambda(1+\gamma) \in [\lambda, x^*]$. Hence
$$
\lambda(1+\gamma)a-c \ge a x^* -c = \frac{p(x^*)}{x^*-\lambda} \ge \frac{p(x^*)}{1-\lambda}.
$$
This completes the proof.
\end{proof}

Lemma~\ref{lemma:a<0} clearly indicates that if a $(\lambda, \gamma)$-admissible concave quadratic polynomial $p$ can attain the value $1$ on $[\lambda, 1]$, then its gap is at least the magnitude of the optimal gap of the $(\lambda, \gamma)$-admissible degree $1$ polynomials. Specifically, for $\lambda \in (1-\frac{\sqrt{2}}{2},\frac{1}{2})$ if we choose $a=\frac{2\lambda-1}{\lambda(1-\lambda)}<0$ and $c=-\frac{1}{\lambda}$ in \eqref{eq:deg2}, then the polynomial $p(x)=(x-\lambda)(\frac{2\lambda-1}{\lambda(1-\lambda)}x+\frac{1}{\lambda})$ is $(\lambda, \gamma)$ admissible and $p(1)=1$. Similarly, if $\lambda \in (0,1-\frac{\sqrt{2}}{2}]$, we can choose $a=\frac{-1}{(1-\lambda)^2}, c= -\frac{2-\lambda}{(1-\lambda)^2}$ and the polynomial $p(x)= (x-\lambda)(\frac{-1}{(1-\lambda)^2}x +\frac{2-\lambda}{(1-\lambda)^2})$ is $(\lambda,\gamma)$-admissible and satisfies $p(1)=1$. Therefore, in either case, $\alpha(p; \lambda, \gamma) \ge \alpha_1(\lambda,\gamma)$ by Lemma~\ref{lemma:a<0}. Hence, from Lemma~\ref{lemma:a>0} and Lemma~\ref{lemma:a<0}, the optimal $(\lambda, \gamma)$-admissible quadratic polynomial must be concave and its gap must be bigger than  the optimal gap of $(\lambda, \gamma)$-admissible polynomial of degree $1$.

In the following discussion, we will search for a $(\lambda, \gamma)$-admissible polynomial which yields the maximum value of $\alpha(p; \lambda, \gamma)$ given in \eqref{eq:gap-a<0}. Since this value is determined by parameters $a$ and $c$ for the given pair $(\lambda,\gamma)$, we consider the following optimization problem
\begin{equation}\label{def:P}
\mathrm{arg}\min_{(a,c) \in \mathcal{S}} P(a,c):=\lambda^2\gamma(1+\gamma)a-\lambda \gamma c,
\end{equation}
where $\mathcal{S}$ is a set that contains all possible pairs of $(a,c)$ that yield $(\lambda,\gamma)$-admissible polynomials. Note that $P$ in \eqref{def:P} is a harmonic function, i.e.,
$$
\frac{\partial^2 P}{\partial a^2} + \frac{\partial^2 P}{\partial c^2}=0.
$$
By the well-known maximal principal (see, e.g., \cite{Evans:98}), if $\mathcal{S}$ is bounded, then the maximum of $P$ can be achieved on its boundary.

To evaluate $P$ on the boundary of the set $\mathcal{S}$, we first investigate the structure of this set. Given the condition $a< 0$, the inequalities of Definition~\ref{def:admissible-poly} imply that $c<0$ from the requirement $p(0)<0$. Similarly, the root $\frac{c}{a}$ must be at least $1+\lambda\gamma$ from the requirement $p(\lambda(1+\gamma)>0$. Since the polynomial $p$ attains its maximum at $\frac{1}{2}(\lambda+\frac{c}{a})$ the middle point of two roots $\lambda$ and $\frac{c}{a}$, we describe the set $\mathcal{S}$ according to the situations $\frac{1}{2}(\lambda+\frac{c}{a})\ge 1$ and $\frac{1}{2}(\lambda+\frac{c}{a})< 1$ separately.

In the rest of this section,  we present the optimal solution to problem~\eqref{def:P} under the situation of $\frac{1}{2}(\lambda+\frac{c}{a})\ge 1$ in subsection~\ref{subsec:1} and the optimal solution to problem~\eqref{def:P} under the situation of $\frac{1}{2}(\lambda+\frac{c}{a})<1$ in  subsection~\ref{subsec:2}. By analyzing the results in these two subsections, we give the explicit expression of the optimal $(\lambda, \gamma)$-admissible polynomial of degree 2 in subsection~\ref{subsec:2}.

\subsection{The Case of $\frac{1}{2}(\lambda+\frac{c}{a})\ge 1$}\label{subsec:1}
Recall the $(\lambda, \gamma)$-admissible polynomial $p$ in \eqref{eq:deg2} with $a < 0$ and $\lambda \in (0, 1/2)$.
The case $\frac{1}{2}(\lambda+\frac{c}{a})\ge 1$ says $c\le (2-\lambda)a$. The condition $-1 \le p(x) \le p((1-\gamma)\lambda) <0$ for all $x\in [0, (1-\gamma)\lambda]$ requires $p(0) \ge -1$ which is $c\ge -\frac{1}{\lambda}$. Finally, the condition $0 \le p((1+\gamma)\lambda)\le p(x) \le 1$ for all $x\in [(1+\gamma)\lambda,1]$ is satisfied if $p(1)\le 1$ which is $c \ge a-\frac{1}{1+\lambda}$.
Hence
\begin{equation}\label{eq:S-case2}
\mathcal{S} = \left\{(a,c): a < 0, \; c\le (2-\lambda)a, \; c\ge -\frac{1}{\lambda}, \; c \ge a-\frac{1}{1+\lambda}\right\}.
\end{equation}
Clearly, the set $\mathcal{S}$ is a subset of the triangular region $\left\{(a,c): a \le 0, \; c\le (2-\lambda)a, \; c\ge -\frac{1}{\lambda}\right\}$. So the set $\mathcal{S}$  is bounded. To find the maximum value of $P$ in \eqref{def:P}, we need to identify the boundary of $\mathcal{S}$.

In the $ac$-plane, the lines $ c=(2-\lambda)a$ and $c = a-\frac{1}{1+\lambda}$ intersect when $a=-\frac{1}{(1-\lambda)^2}$ while the lines $ c=(2-\lambda)a$ and $c=-\frac{1}{\lambda}$ meet when $a=-\frac{1}{2-\lambda}$. Hence, these lines have a common intersection point when $\frac{1}{(1-\lambda)^2}=\frac{1}{2-\lambda}$, that is, $\lambda=1-\frac{\sqrt{2}}{2}$. As depicted in Figure~\ref{fig:2}, the geometric shape of $\mathcal{S}$ is a triangle for $\lambda \in (0, 1-\frac{\sqrt{2}}{2}]$ and a quadrilateral for $\lambda \in (1-\frac{\sqrt{2}}{2}, \frac{1}{2})$.

\begin{figure}[htpb]
 \centering
 \begin{tabular}{cc}
 \includegraphics[scale=0.40]{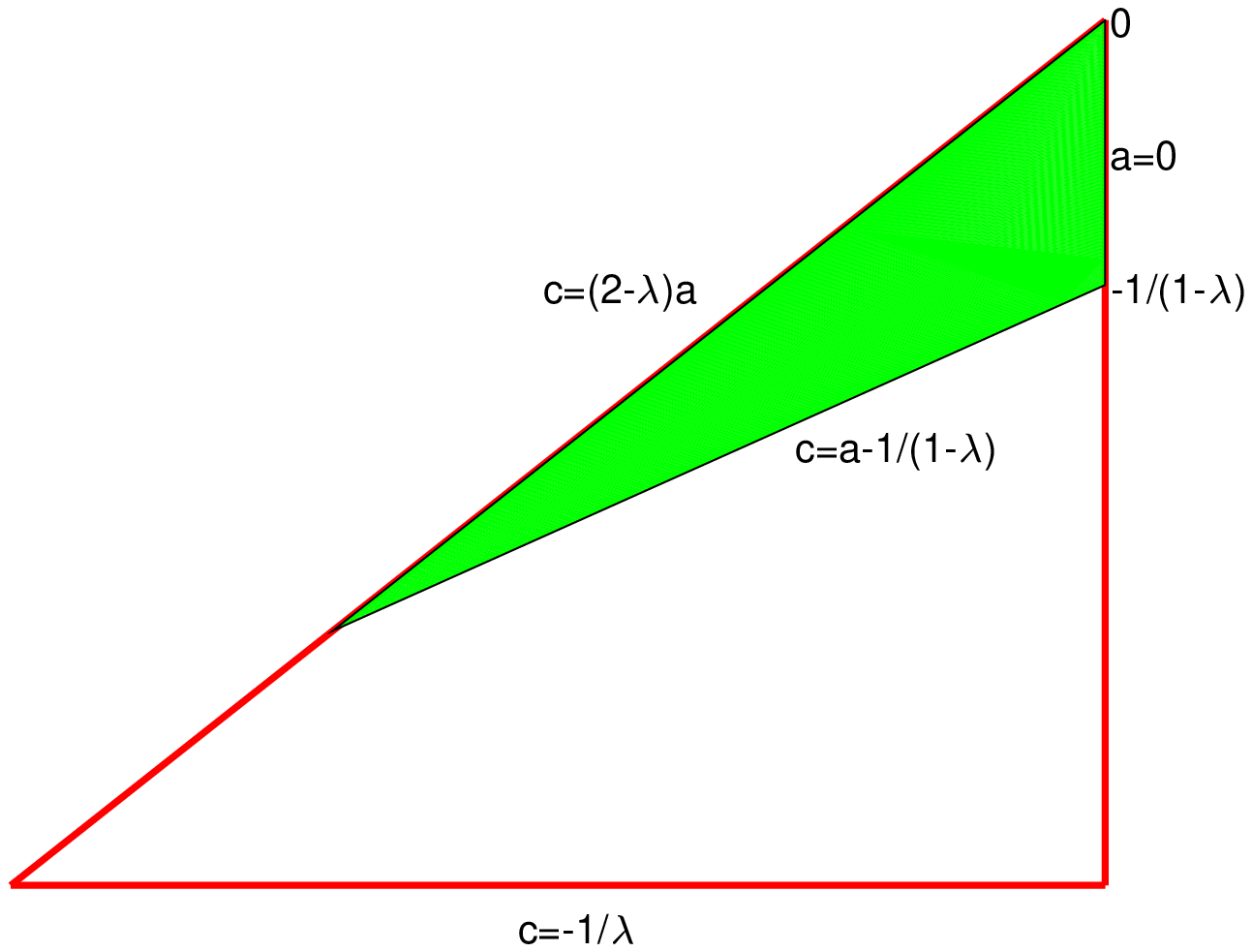}&
\includegraphics[scale=0.40]{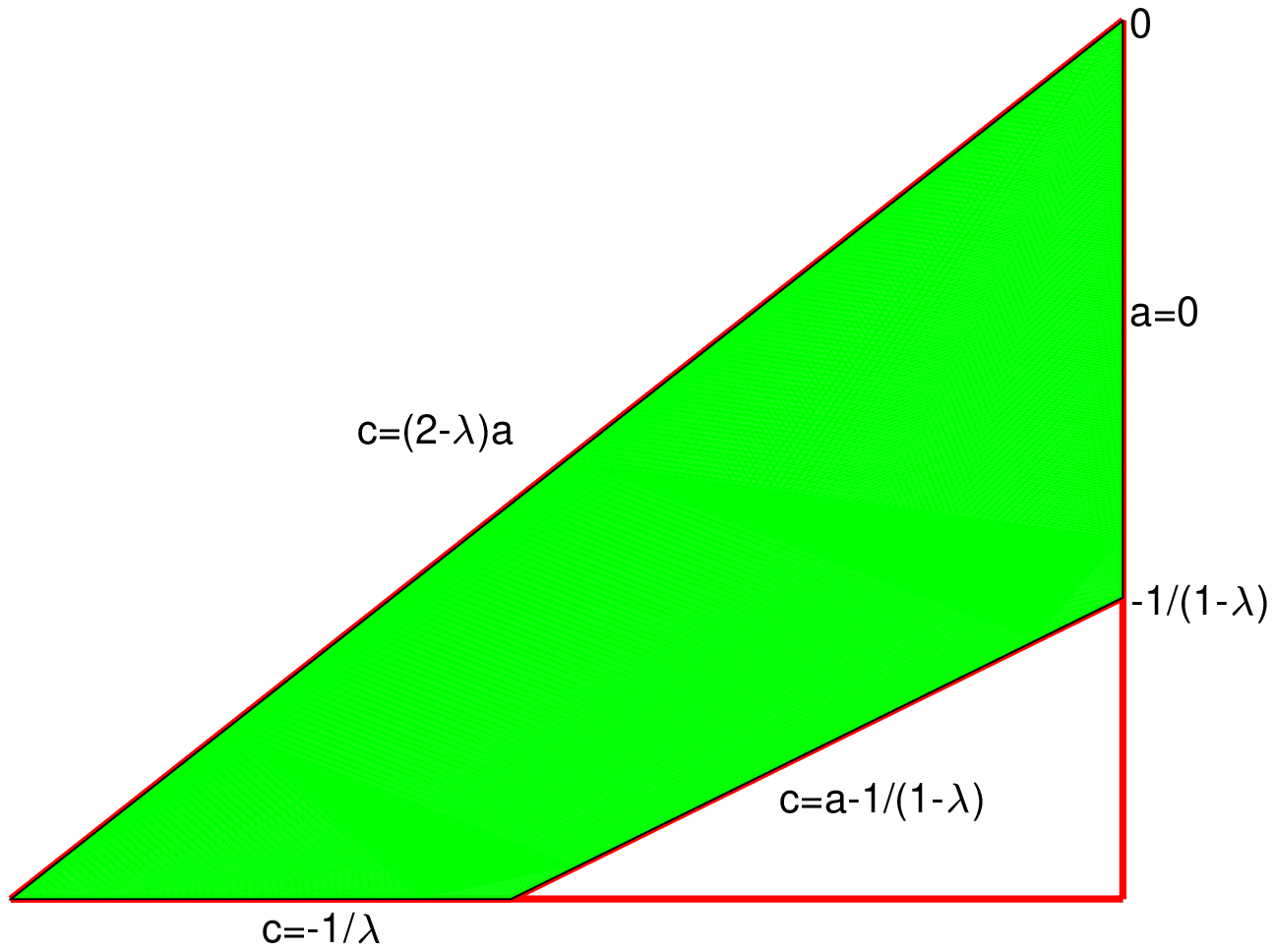}\\
(a) $\lambda \in (0, 1-\frac{\sqrt{2}}{2}]$ &(b) $\lambda \in (1-\frac{\sqrt{2}}{2}, \frac{1}{2})$
\end{tabular}
\caption{The plots of $\mathcal{S}$ when (a) $\lambda \in (0, 1-\frac{\sqrt{2}}{2}]$ and (b) $\lambda \in (1-\frac{\sqrt{2}}{2}, \frac{1}{2})$.}
\label{fig:2}
\end{figure}

\begin{theorem}\label{thm:case2}
Let $\lambda \in (0,1/2)$ and $\gamma \in (0,1)$. If the pair $(a^*,c^*)$ is the solution to problem~\eqref{def:P}, where $\mathcal{S}$ is given in \eqref{eq:S-case2}, then the following statements hold.
\begin{enumerate}
\item[(i)] If $\lambda \in (0, 1-\frac{\sqrt{2}}{2}]$, the optimal solution to the optimization problem is 
$$(
a^*,c^*)=\left(-\frac{1}{(1-\lambda)^2}, -\frac{2-\lambda}{(1-\lambda)^2}\right)
$$ 
and the optimal is $P(a^*,c^*)=\frac{2\lambda \gamma}{1-\lambda}-\frac{(\lambda \gamma)^2}{(1-\lambda)^2}$.

\item[(ii)] If $\lambda \in (1-\frac{\sqrt{2}}{2}, \frac{1}{2})$, the optimal solution to the optimization  problem  is 
$$
(a^*,c^*)=\left(\frac{2\lambda-1}{(1-\lambda)\lambda}, -\frac{1}{\lambda}\right)
$$ 
and the optimal is $P(a^*,c^*)=\frac{\lambda(2\lambda-1)\gamma(1+\gamma)}{1-\lambda}+\gamma$.
\end{enumerate}
\end{theorem}
\begin{proof}
(i) $\lambda \in (0, 1-\frac{\sqrt{2}}{2}]$: As we can see from Figure~\ref{fig:2}(a), the boundary of $\mathcal{S}$ consists of three line segments: $S_1:=\{(0,c): c \in [0, -1/(1-\lambda]\}$, $S_2:=\{(a,c): a \in [-1/(1-\lambda)^2,0], \; c=a-1/(1-\lambda)\}$, and
$S_3:=\{(a,c): a \in [-1/(1-\lambda)^2,0], \; c=(2-\lambda)a\}$. It can be directly verified that $P$ has the maximum value $\frac{\lambda\gamma}{1-\lambda}$ on $S_1$ at the point $(0, -1/(1-\lambda)$ and the maximum value of $\frac{2\lambda \gamma}{1-\lambda}-\frac{(\lambda \gamma)^2}{(1-\lambda)^2}$ on both $S_2$ and $S_3$ at the point $(-1/(1-\lambda)^2, -(2-\lambda)/(1-\lambda)^2)$. Since $\frac{\lambda \gamma}{1-\lambda}<1$, then $\frac{2\lambda \gamma}{1-\lambda}-\frac{(\lambda \gamma)^2}{(1-\lambda)^2}>\frac{\lambda\gamma}{1-\lambda}$, thus, the statement holds.

(ii) $\lambda \in (1-\frac{\sqrt{2}}{2}, \frac{1}{2})$:  As we can see from Figure~\ref{fig:2}(b), the boundary of $\mathcal{S}$ consists of four line segments: $S_1:=\{(0,c): c \in [0, -1/(1-\lambda]\}$, $S_2:=\{(a,c): a \in [\frac{2\lambda-1}{(1-\lambda)\lambda},0], \; c=a-1/(1-\lambda)\}$, $S_3:=\{(a,c): a \in [-\frac{1}{\lambda(2-\lambda)}, \frac{2\lambda-1}{(1-\lambda)\lambda}], \; c=-\frac{1}{\lambda}\}$, and $S_4:=\{(a,c): a \in [-\frac{1}{\lambda(2-\lambda)}, 0], \; c=(2-\lambda)a\}$. It can be checked easily that $P$ attains the maximal value $\frac{\lambda(2\lambda-1)\gamma(1+\gamma)}{1-\lambda}+\gamma$ on $S_1 \cup S_2\cup S_3 \cup S_4$ at the point $\left(\frac{2\lambda-1}{(1-\lambda)\lambda}, -\frac{1}{\lambda}\right)$.
\end{proof}
Typical graphs of the $(\lambda,\gamma)$-admissible polynomial $p(x)=(x-\lambda)(a^*x-c^*)$ for $(a^*,c^*)$ given in Theorem~\ref{thm:case2} are depicted in  Figure~\ref{fig:AAcase-polynomial}. For the purpose of comparison, the graph of the optimal $(\lambda,\gamma)$-admissible linear polynomial $p^{\lambda,\gamma}_1$ is included in Figure~\ref{fig:AAcase-polynomial} as well. We remark that if we replace the pair $(a,c)$ in \eqref{eq:deg2} by each pair $(a^*,c^*)$ from Theorem~\ref{thm:case2}, then we always have $p(1)=(1-\lambda)(a^*-c^*)=1$.

\begin{figure}[htpb]
 \centering
 \begin{tabular}{cc}
 \includegraphics[scale=0.40]{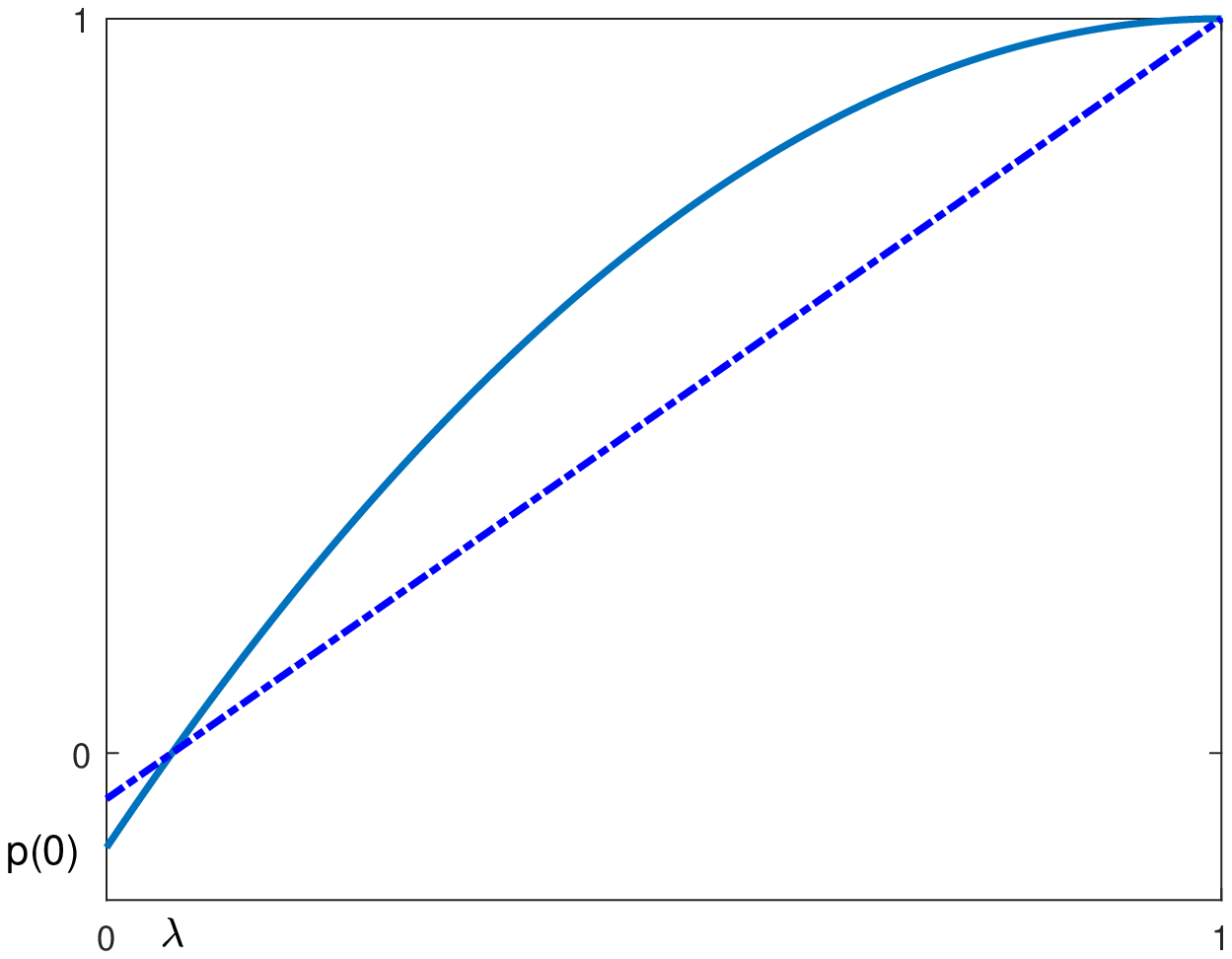}&
\includegraphics[scale=0.40]{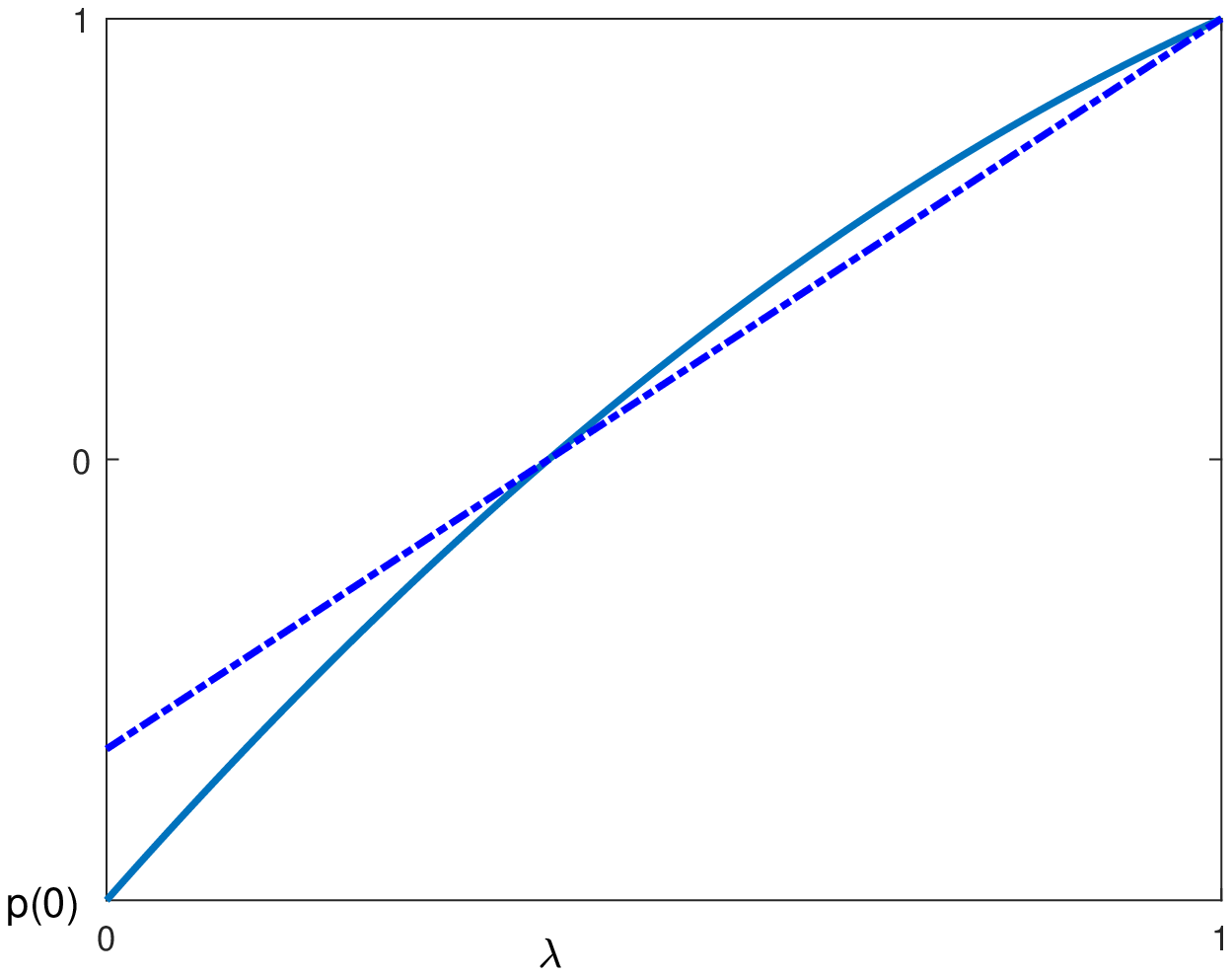}\\
(a)&(b)
\end{tabular}
\caption{The plots of the function $p(x)=(x-\lambda)(a^*x-c^*)$ (solid line), where the pair $(a^*, c^*)$ is given in Theorem~\ref{thm:case2}, and the optimal $(\lambda,\gamma)$-admissible linear polynomial $p^{\lambda,\gamma}_1(x)=\frac{x-\lambda}{1-\lambda}$ on the interval $[0,1]$ for (a) $\lambda \in (0, 1-\frac{\sqrt{2}}{2}]$, $\gamma \in (0,1)$; and (b) $\lambda \in [1-\frac{\sqrt{2}}{2}, \frac{1}{2}]$, $\gamma \in (0,1)$.}
\label{fig:AAcase-polynomial}
\end{figure}

\subsection{The Case of $\frac{1}{2}(\lambda+\frac{c}{a})< 1$}\label{subsec:2}
Recall the $(\lambda, \gamma)$-admissible polynomial $p$ in \eqref{eq:deg2} with $a < 0$ and $\lambda \in (0, 1/2)$.
The case $\frac{1}{2}(\lambda+\frac{c}{a})< 1$ says $c> (2-\lambda)a$. The condition $-1 \le p(x) \le p((1-\gamma)\lambda) <0$ for all $x\in [0, (1-\gamma)\lambda]$ implies $p(0)\ge -1$ which is equivalent to $c \ge -\frac{1}{\lambda}$. The condition $0 \le p((1+\gamma)\lambda)\le p(x) \le 1$ for all $x\in [(1+\gamma)\lambda,1]$ indicates that both $p(1) \ge p(\lambda(1+\gamma))$ and $p\left(\frac{1}{2}(\lambda+\frac{c}{a})\right) \le 1$ hold.  The inequality $p(1) \ge p(\lambda(1+\gamma))$ holds if and only if $c \le (1+\lambda \gamma)$. The inequality $p\left(\frac{1}{2}(\lambda+\frac{c}{a})\right) \le 1$ leads to $c \ge \lambda a -2\sqrt{-a}$. Putting all these together, the set $\mathcal{S}$, which all possible pairs of $(a,c)$ belong to, is characterized as follows:
\begin{equation}\label{eq:S-case1}
\mathcal{S} = \left\{(a,c): a < 0, \; c > (2-\lambda)a, \;  c \ge -\frac{1}{\lambda}, \; c \le (1+\lambda \gamma), \; c \ge \lambda a -2\sqrt{-a}\right\}.
\end{equation}
Clearly, the set $\mathcal{S}$ is bounded since it is  a subset of the triangular region
\[\left\{(a,c): a < 0, \; c > (2-\lambda)a, \;  c \ge -\frac{1}{\lambda}, \; c \le (1+\lambda \gamma)\right\}.\]
To find the maximum value of $P$ in \eqref{def:P}, we need to identify the boundary of $\mathcal{S}$.

The boundary of $\mathcal{S}$ can be seen easily if we know the detailed geometric structures of this set. To this end, we first define three lines: $L_1$, $L_2$, and $L_3$ in the $ac$-plane via $c=(2-\lambda)a$, $c=-\frac{1}{\lambda}$, and $c=(1+\lambda \gamma)a$, respectively. We denote the value of $a$ where the lines $L_1$ and $L_3$ intersect with $L_2$ as $a_{(L_1,L_2)}$ and $a_{(L_3,L_2)}$, respectively. Then, $a_{(L_1,L_2)}=-\frac{1}{\lambda(2-\lambda)}$ and $a_{(L_3,L_2)}=-\frac{1}{\lambda(1+\lambda\gamma)}$. Note that $a_{(L_1,L_2)}> a_{(L_3,L_2)}$ for all $\lambda \in (0,1/2)$ and $\gamma \in (0,1)$. Let $C$ be the curve in $ac$-plane defined by the function $f(a)=\lambda a -2\sqrt{-a}$ which is increasing and convex on $(-\infty, 0]$. This curve intersects with line $L_2$ at the value $a_{(C,L_2)}=\frac{2\sqrt{2}-3}{\lambda^2}$. The geometric structure of $\mathcal{S}$ is dependent on the value of $a_{(C,L_2)}$ in relation to both $a_{(L_1,L_2)}$ and $a_{(L_3,L_2)}$. Thus there are three separate cases to consider:
\[a_{(C,L_2)} \leq a_{(L_3,L_2)}, \quad a_{(L_3,L_2)} \le a_{(C,L_2)} \le a_{(L_1,L_2)},\quad \text{and } a_{(L_1,L_2)}\le a_{(C,L_2)}.\]
These inequalities require $\lambda \in (0,\frac{1}{3+2\sqrt{2}-\gamma}]$, $\lambda \in [\frac{1}{3+2\sqrt{2}-\gamma}, 1-\frac{\sqrt{2}}{2}]$, and $\lambda \in [1-\frac{\sqrt{2}}{2}, \frac{1}{2})$, respectively. The geometric structures of $\mathcal{S}$ for three different cases are plotted in Figure~\ref{fig:1}.

\begin{figure}[htpb]
 \centering
 \begin{tabular}{ccc}
 \includegraphics[scale=0.40]{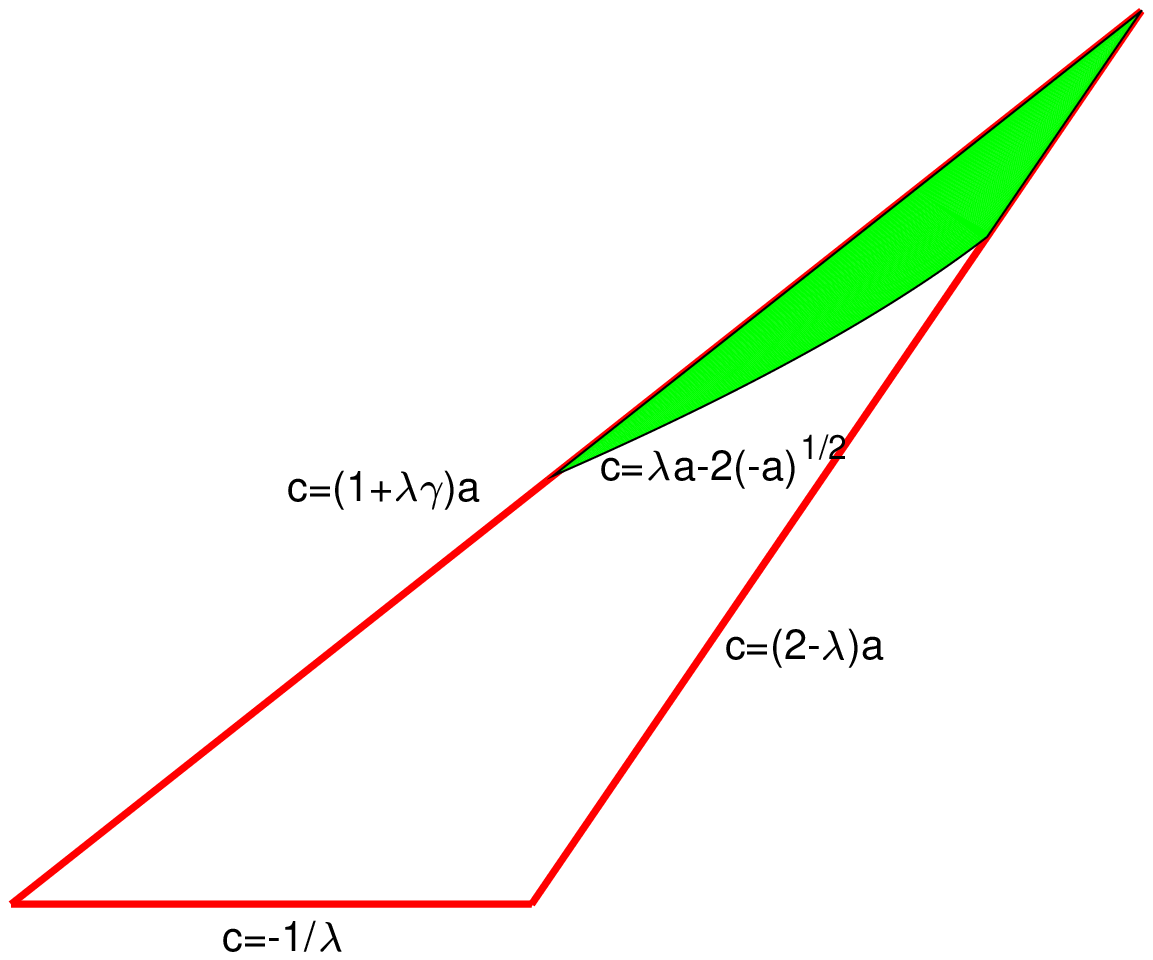}&
\includegraphics[scale=0.40]{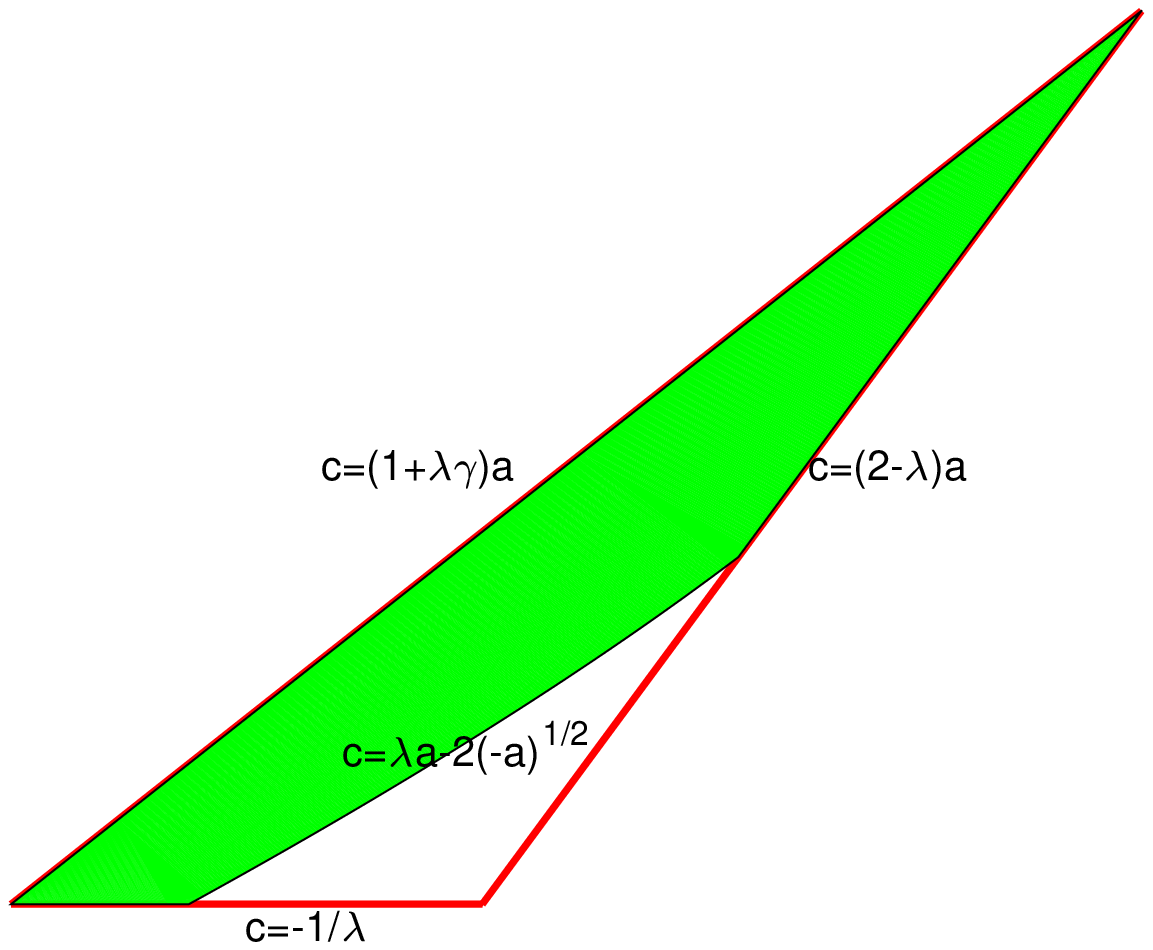}&
\includegraphics[scale=0.40]{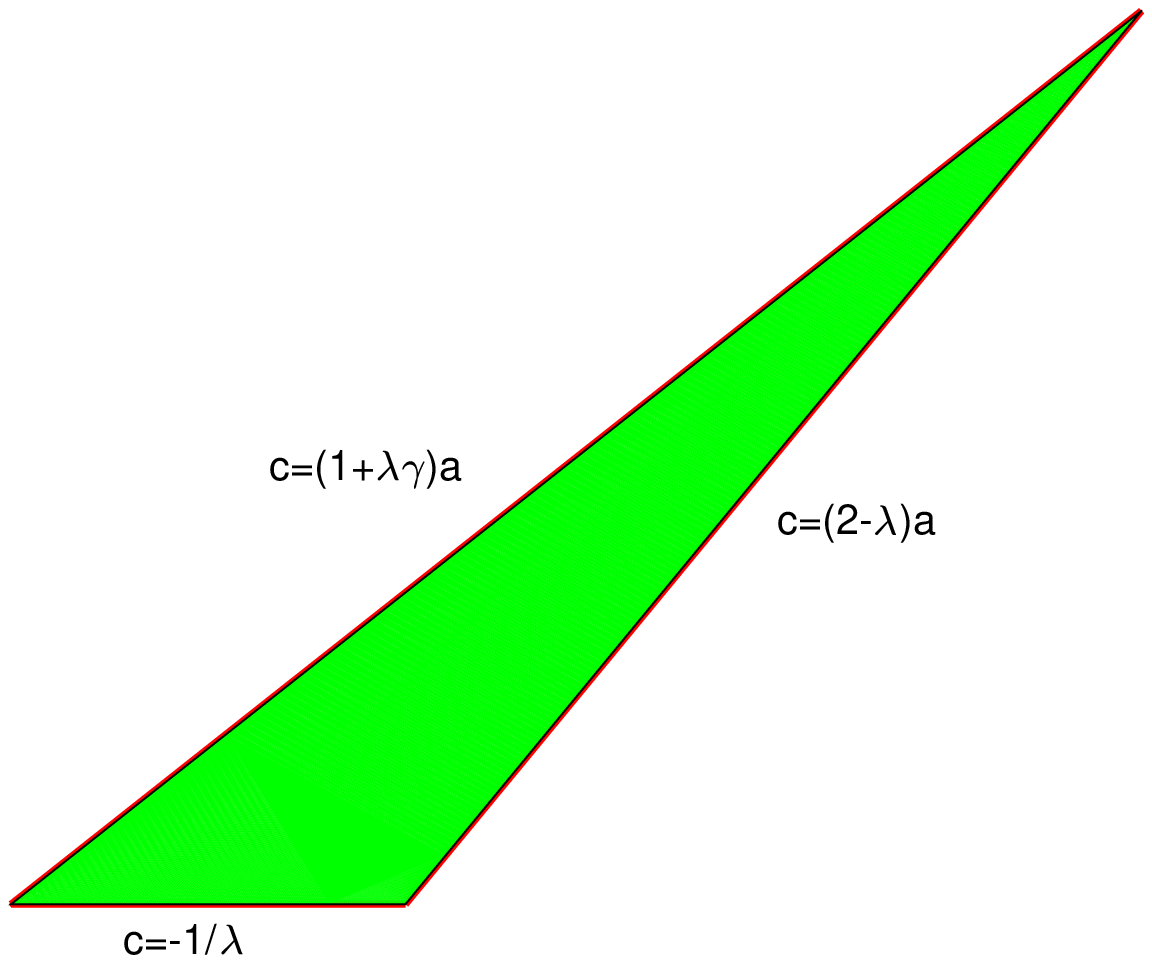}\\
(a) $\lambda \in (0,\frac{1}{3+2\sqrt{2}-\gamma}]$ & (b) $\lambda \in [\frac{1}{3+2\sqrt{2}-\gamma}, 1-\frac{\sqrt{2}}{2}]$; & (c) $\lambda \in [1-\frac{\sqrt{2}}{2}, \frac{1}{2})$
\end{tabular}
\caption{The plots of $\mathcal{S}$ when (a) $\lambda \in (0,\frac{1}{3+2\sqrt{2}-\gamma}]$; (b) $\lambda \in [\frac{1}{3+2\sqrt{2}-\gamma}, 1-\frac{\sqrt{2}}{2}]$;  and (c) $\lambda \in [1-\frac{\sqrt{2}}{2}, \frac{1}{2})$.}
\label{fig:1}
\end{figure}

\begin{theorem} \label{thm:case1}
Let $\lambda \in (0,1/2)$ and $\gamma \in (0,1)$. If the pair $(a^*,c^*)$ is the solution to problem~\eqref{def:P}, where $\mathcal{S}$ is given in \eqref{eq:S-case1}, then the following statements hold.
\begin{enumerate}
\item[(i)] If $\lambda \in (0,\frac{1}{3+2\sqrt{2}-\gamma}]$, the optimal solution to the optimization problem is 
$$
    (a^*,c^*)=\left(\frac{-4}{(1-\lambda+\lambda\gamma)^2}, \frac{-4(1+\lambda\gamma)}{(1-\lambda+\lambda\gamma)^2}\right)
$$ 
and the optimal is $P(a^*,c^*)=\frac{4\gamma\lambda(1-\lambda)}{(1+\lambda\gamma-\lambda)^2}$.

\item[(ii)] If $\lambda \in [\frac{1}{3+2\sqrt{2}-\gamma}, 1-\frac{\sqrt{2}}{2}]$, the optimal solution to the optimization  problem is 
$$
(a^*,c^*)=\left(\frac{-1}{(3+2\sqrt{2})\lambda^2}, -\frac{1}{\lambda}\right)
$$ 
and the optimal is $P(a^*,c^*)=\frac{\gamma}{ (3+2\sqrt{2})}(2+2\sqrt{2}-\gamma )$.

\item[(iii)] If $\lambda \in [1-\frac{\sqrt{2}}{2}, \frac{1}{2})$, the optimal solution to the optimization  problem is
$$
    (a^*,c^*)=\left(\frac{-1}{(2-\lambda)\lambda}, -\frac{1}{\lambda}\right)
$$ 
and the optimal is $P(a^*,c^*)=\frac{\gamma(2-2\lambda-\lambda\gamma )}{(2-\lambda)}$.
\end{enumerate}
\end{theorem}
\begin{proof}
(i) $\lambda \in (0,\frac{1}{3+2\sqrt{2}-\gamma}]$: As we can see from Figure~\ref{fig:1}(a), the boundary of $\mathcal{S}$ consists of three segments:
\begin{eqnarray*}
S_1&:=&\left\{(a,(2-\lambda)a): a \in [-\frac{1}{(1-\lambda)^2},0]\right\}, \\
S_2&:=&\left\{(a,\lambda a -2\sqrt{-a}): a \in [-\frac{4}{(1-\lambda+\lambda\gamma)^2},-\frac{1}{(1-\lambda)^2}]\right\},\\
S_3&:=&\left\{(a,(1+\lambda \gamma)a): a \in [-\frac{4}{(1-\lambda+\lambda\gamma)^2},0]\right\}.
\end{eqnarray*}
By verifying the maximum value of $P$ on $S_1$, $S_2$, and $S_3$, we know the conclusion of part 1 holds.

(ii) $\lambda \in [\frac{1}{3+2\sqrt{2}-\gamma}, 1-\frac{\sqrt{2}}{2}]$:  From Figure ~\ref{fig:1}(b), the boundary of $\mathcal{S}$ consists of four segments:
\begin{eqnarray*}
S_1&:=&\left\{(a,(2-\lambda)a): a \in [-\frac{1}{(1-\lambda)^2},0]\right\},\\
S_2&:=&\left\{(a,\lambda a -2\sqrt{-a}): a \in [\frac{2\sqrt{2}-3}{\lambda^2},-\frac{1}{(1-\lambda)^2}]\right\},\\
S_3&:=&\left\{(a,-\frac{1}{\lambda}): a \in [-\frac{1}{\lambda(1+\lambda\gamma)},\frac{2\sqrt{2}-3}{\lambda^2}]\right\},\\
S_4&:=&\left\{(a,(1+\lambda \gamma)a): a \in [-\frac{1}{\lambda(1+\lambda\gamma)},0]\right\}.
\end{eqnarray*}
By verifying the maximum value of $P$ on $S_1$, $S_2$, $S_3$, and $S_4$, we know the conclusion of part 2 holds.

(iii) $\lambda \in [1-\frac{\sqrt{2}}{2}, \frac{1}{2})$: As we can see from Figure~\ref{fig:1}(c), the boundary of $\mathcal{S}$ consists of three segments: \begin{eqnarray*}
S_1&:=&\left\{(a,(2-\lambda)a): a \in [-\frac{1}{\lambda(2-\lambda)},0]\right\},\\
S_2&:=&\left\{(a,-\frac{1}{\lambda}): a \in [-\frac{1}{\lambda(1+\lambda\gamma)},-\frac{1}{\lambda(2-\lambda)}]\right\},\\
S_3&:=&\left\{(a,(1+\lambda \gamma)a): a \in [-\frac{1}{\lambda(1+\lambda\gamma)},0]\right\}.
\end{eqnarray*}
By verifying the maximum value of $P$ on $S_1$, $S_2$, and $S_3$, we know the conclusion of part 3 holds.
\end{proof}

We remark that if we replace the pair $(a,c)$ in \eqref{eq:deg2} by each first two pair $(a^*,c^*)$ from Theorem~\ref{thm:case1}, then we always have $p(\frac{1}{2}(\lambda+\frac{c^*}{a^*}))=1$. But for the third pair $(a^*,c^*)=\left(\frac{-1}{(2-\lambda)\lambda}, -\frac{1}{\lambda}\right)$, we have $p(\frac{1}{2}(\lambda+\frac{c^*}{a^*}))=\frac{(1-\lambda)^2}{\lambda(2-\lambda)}$ which is strictly less than $1$ for $\lambda \in (1-\frac{\sqrt{2}}{2}, \frac{1}{2})$. Typical graphs of the $(\lambda,\gamma)$-admissible polynomial $p(x)=(x-\lambda)(a^*x-c^*)$ for $(a^*,c^*)$ given in Theorem~\ref{thm:case1} are depicted in  Figure~\ref{fig:acase-polynomial}. For the purpose of comparison, the graph of the optimal $(\lambda,\gamma)$-admissible linear polynomial $p^{\lambda,\gamma}_1$ is included in Figure~\ref{fig:acase-polynomial} as well.

\begin{figure}[htpb]
 \centering
 \begin{tabular}{ccc}
\includegraphics[scale=0.40]{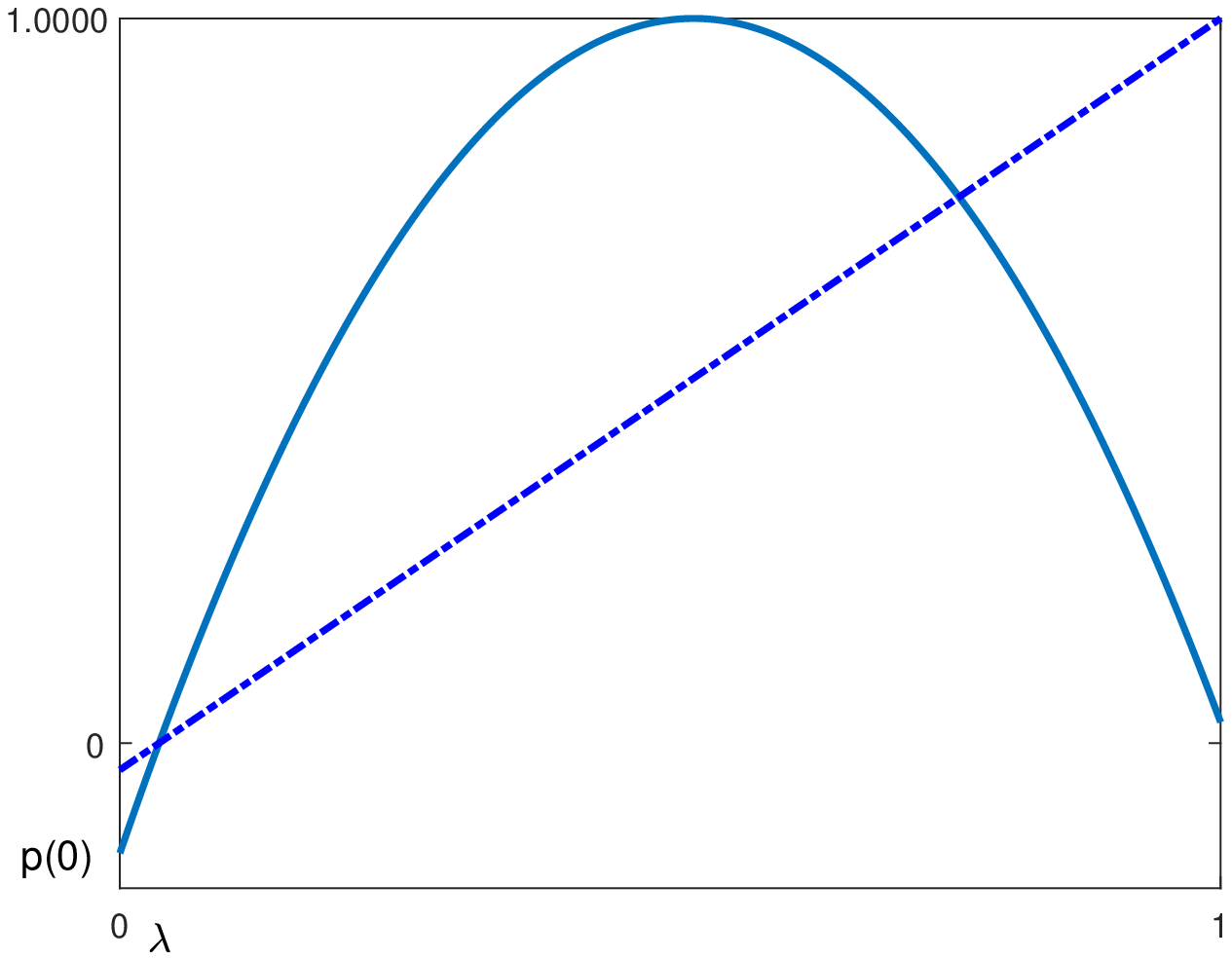}&
\includegraphics[scale=0.40]{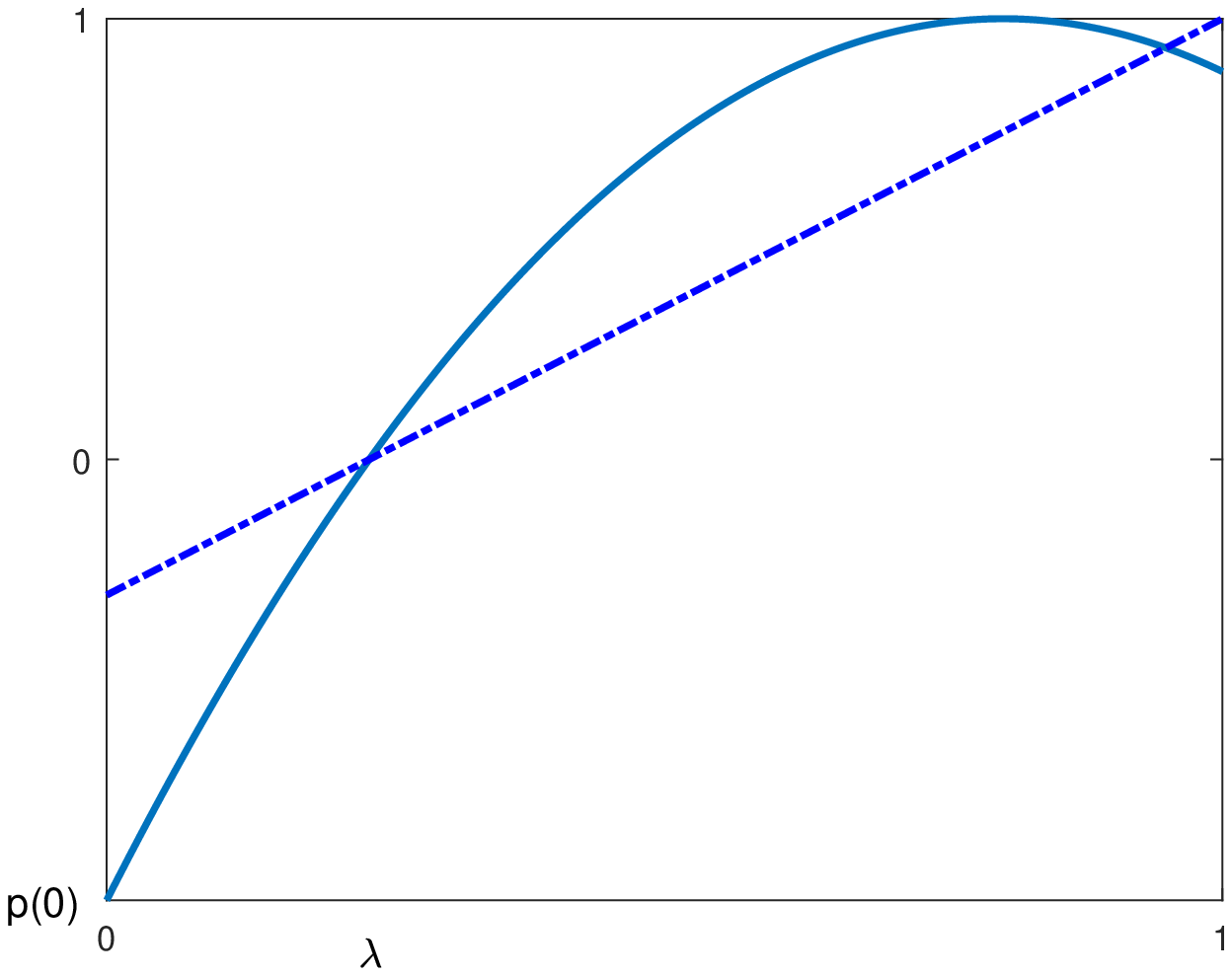}&
\includegraphics[scale=0.40]{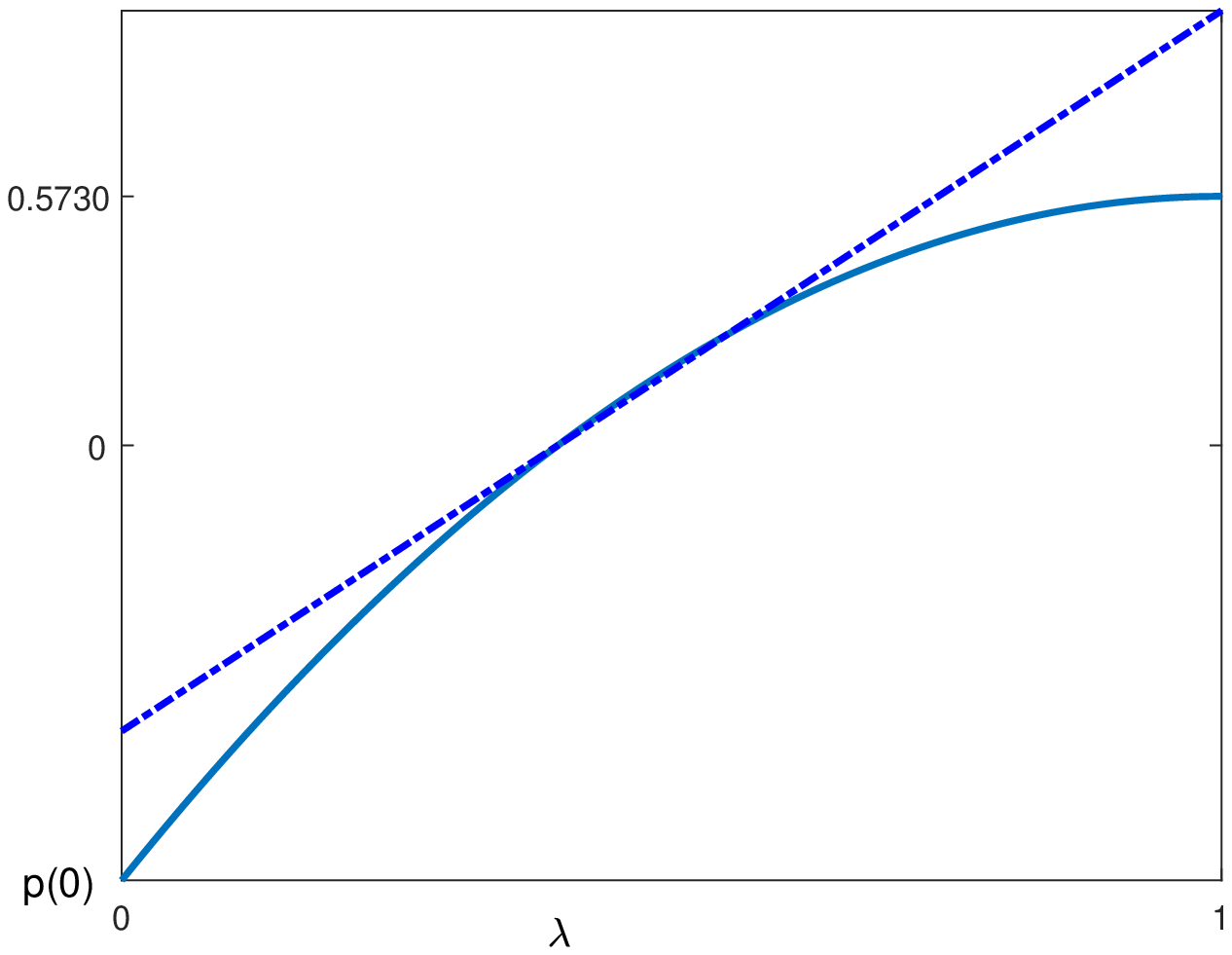}\\
(a)&(b)&(c)
\end{tabular}
\caption{The plots of the function $p(x)=(x-\lambda)(a^*x-c^*)$ (solid line), where the pair $(a^*, c^*)$ is given in Theorem~\ref{thm:case1},  and the optimal $(\lambda,\gamma)$-admissible linear polynomial $p^{\lambda,\gamma}_1(x)=\frac{x-\lambda}{1-\lambda}$ on the interval $[0,1]$ for (a) $\lambda \in (0,\frac{1}{3+2\sqrt{2}-\gamma}]$, $\gamma \in (0,1)$; (b) $\lambda \in [\frac{1}{3+2\sqrt{2}-\gamma}, 1-\frac{\sqrt{2}}{2}]$, $\gamma \in (0,1)$ and (c) $\lambda \in [1-\frac{\sqrt{2}}{2}, \frac{1}{2})$, $\gamma \in (0,1)$.}
\label{fig:acase-polynomial}
\end{figure}

\subsection{The Optimal Admissible Quadratic Polynomials}\label{subsec:3}

In the previous two subsections, for each value of $\lambda$ in the interval $(0,\frac{1}{2})$ we found the solution $(a^*, c^*)$ that can maximize the objective  function $P$ in \eqref{def:P} under the constraints $\frac{1}{2}(\lambda+\frac{c}{a})\ge 1$ and $\frac{1}{2}(\lambda+\frac{c}{a})< 1$. By comparing the corresponding optimal values of $P$, we are able to construct the optimal admissible quadratic polynomial for each given pair $(\lambda, \gamma)$ with $\lambda \in (0, 1/2)$.

The pair $(a^*, c^*)$ in Theorem~\ref{thm:case2} and Theorem~\ref{thm:case1} can be viewed as a function of $\lambda$ and its explicit form changes according to the interval which $\lambda$ belongs to. Specifically, the form of $(a^*, c^*)$ from Theorem~\ref{thm:case2} changes at $1-\frac{1}{2}\sqrt{2}$ while the form of $(a^*, c^*)$ in Theorem~\ref{thm:case1} changes at $\frac{1}{3+2\sqrt{2}-\gamma}$ and $1-\frac{1}{2}\sqrt{2}$. We, therefore, compare the values of $P(a^*,c^*)$ in Theorem~\ref{thm:case2} and Theorem~\ref{thm:case1} for $\lambda$ in three different intervals, namely $(0,\frac{1}{3+2\sqrt{2}-\gamma}]$,  $[\frac{1}{3+2\sqrt{2}-\gamma}, 1-\frac{\sqrt{2}}{2}]$, and $[1-\frac{\sqrt{2}}{2}, \frac{1}{2})$.

\begin{proposition}\label{P2>P1}
For $\lambda \in (0,1/2)$ and $\gamma \in (0,1)$, define
$$
Q(\lambda)=P(a_2^*,c_2^*)-P(a_1^*, c_1^*),
$$
where $(a_1^*, c_1^*)$ and $(a_2^*,c_2^*)$, from Theorem~\ref{thm:case2} and Theorem~\ref{thm:case1} correspondingly, are the optimal solutions maximizing the objective function $P$ in \eqref{def:P}. Then, we have
$$
Q(\lambda)\left\{
            \begin{array}{ll}
              >0, & \hbox{if $\lambda \in (0,\frac{1}{3+2\sqrt{2}-\gamma}]$;} \\
              \ge 0, & \hbox{if $\lambda \in [\frac{1}{3+2\sqrt{2}-\gamma}, 1-\frac{\sqrt{2}}{2}]$;} \\
              \le 0, & \hbox{if $\lambda \in [1-\frac{\sqrt{2}}{2}, \frac{1}{2})$.}
            \end{array}
          \right.
$$
\end{proposition}
\begin{proof}\ \
(i) $\lambda \in (0,\frac{1}{3+2\sqrt{2}-\gamma}]$: In this case, $P(a_1^*,c_1^*)=\frac{\lambda\gamma(2-2\lambda -\lambda\gamma)}{(1-\lambda)^2}$ from Theorem~\ref{thm:case2} and $P(a_2^*,c_2^*)=\frac{4\gamma\lambda(1-\lambda)}{(1+\lambda\gamma-\lambda)^2}$ from Theorem~\ref{thm:case1}. After some algebraic simplifications, we have
$$
Q(\lambda)=P(a_2^*,c_2^*)-P(a_1^*, c_1^*)=2\lambda \gamma \left(\frac{\sqrt{2}}{1+\lambda\gamma-\lambda}+\frac{1}{1-\lambda}\right)\cdot \frac{\sqrt{2}-1+(1-\sqrt{2}-\gamma)\lambda}{(1+\lambda\gamma-\lambda)(1-\lambda)} + \frac{\lambda^2\gamma^2}{(1-\lambda)^2}.
$$
One can verify directly that $\sqrt{2}-1+(1-\sqrt{2}-\gamma)\lambda \ge \frac{2-\sqrt{2}\gamma}{3+2\sqrt{2}-\gamma}>0$ for $\lambda \in (0,\frac{1}{3+2\sqrt{2}-\gamma}]$. Hence, $Q(\lambda)$ is positive for all $\lambda \in (0,\frac{1}{3+2\sqrt{2}-\gamma}]$.

(ii) $\lambda \in [\frac{1}{3+2\sqrt{2}-\gamma}, 1-\frac{\sqrt{2}}{2}]$: Because $P(a_1^*,c_1^*)=\frac{\lambda\gamma(2-2\lambda -\lambda\gamma)}{(1-\lambda)^2}$ from Theorem~\ref{thm:case2} and $P(a_2^*,c_2^*)=  \frac{\gamma((2+2\sqrt{2})-\gamma)}{ (3+2\sqrt{2})} $ from Theorem~\ref{thm:case1}, we have that
$$
Q(\lambda)=\frac{\gamma((2+2\sqrt{2})-\gamma)}{ (3+2\sqrt{2})}-\frac{\lambda\gamma(2-2\lambda -\lambda\gamma)}{(1-\lambda)^2}.
$$
Immediately, we can check that $Q(1-\frac{\sqrt{2}}{2})=0$ and $Q'(\lambda)=\frac{-2\gamma}{(1-\lambda)^3}(1-\lambda-\gamma\lambda)<0$ for all $\lambda \in [\frac{1}{3+2\sqrt{2}-\gamma}, 1-\frac{\sqrt{2}}{2}]$, hence, $Q(\lambda)>0$ for $\lambda \in [\frac{1}{3+2\sqrt{2}-\gamma}, 1-\frac{\sqrt{2}}{2})$.

(iii) $\lambda \in [1-\frac{\sqrt{2}}{2}, \frac{1}{2})$: We have $P(a_1^*,c_1^*)=\frac{((2\lambda-1)(\lambda +\lambda\gamma)+1-\lambda)\gamma}{(1-\lambda)}$ Theorem~\ref{thm:case2} and $P(a_2^*,c_2^*) =\frac{\gamma(2-2\lambda-\lambda\gamma)}{(2-\lambda)}$ from Theorem~\ref{thm:case1}. Through some algebraic simplifications, we get
$$
Q(\lambda)=\frac{(1+\gamma)\lambda(2(1-\lambda)^2-1)}{(2-\lambda)(1-\lambda)}\le 0,
$$
for all $\lambda \in [1-\frac{\sqrt{2}}{2}, \frac{1}{2})$. This completes the proof.
\end{proof}

From the proof of Proposition~\ref{P2>P1}, we know that $Q(\lambda)=0$ holds at $\lambda=1-\frac{\sqrt{2}}{2}$ only.

With this preparation, we are ready to present the optimal $(\lambda, \gamma)$-admissible quadratic polynomial and the associated the optimal gap.
\begin{theorem} \label{Thm:Optimal2}
Let $\lambda \in (0,1/2)$ and $\gamma \in (0,1)$. Then the  optimal $(\lambda, \gamma)$-admissible quadratic polynomial $p_2^{\lambda,\gamma}$ is
$$
p_2^{\lambda,\gamma}(x)=\left\{
            \begin{array}{ll}
              (x-\lambda)\left(\frac{-4}{(1-\lambda+\lambda\gamma)^2} x + \frac{4(1+\lambda\gamma)}{(1-\lambda+\lambda\gamma)^2}\right), & \hbox{if $\lambda \in (0,\frac{1}{3+2\sqrt{2}-\gamma}]$;} \\
              (x-\lambda)\left(\frac{-1}{(3+2\sqrt{2})\lambda^2}x +\frac{1}{\lambda}\right), & \hbox{if $\lambda \in [\frac{1}{3+2\sqrt{2}-\gamma}, 1-\frac{\sqrt{2}}{2}]$;} \\
              (x-\lambda) \left( \frac{2\lambda-1}{(1-\lambda)\lambda} x +\frac{1}{\lambda}\right), & \hbox{if $\lambda \in [1-\frac{\sqrt{2}}{2}, \frac{1}{2})$.}
            \end{array}
          \right.
$$
and the associated optimal gap is
$$
\alpha_2(\lambda,\gamma)=\left\{
            \begin{array}{ll}
              \frac{4\gamma\lambda(1-\lambda)}{(1+\lambda\gamma-\lambda)^2}, & \hbox{if $\lambda \in (0,\frac{1}{3+2\sqrt{2}-\gamma}]$;} \\
              \frac{\gamma}{ (3+2\sqrt{2})}(2+2\sqrt{2}-\gamma), & \hbox{if $\lambda \in [\frac{1}{3+2\sqrt{2}-\gamma}, 1-\frac{\sqrt{2}}{2}]$;} \\
              \frac{\lambda(2\lambda-1)\gamma(1+\gamma)}{1-\lambda}+\gamma, & \hbox{if $\lambda \in [1-\frac{\sqrt{2}}{2}, \frac{1}{2})$.}
            \end{array}
          \right.
$$
\end{theorem}
\begin{proof}\ \
We know by definition that $p_2^{\lambda,\gamma}$ has form \eqref{eq:deg2}. By invoking Proposition~\ref{P2>P1}, Lemma~\ref{lemma:a>0} and Lemma~\ref{lemma:a<0}, the above $p_2^{\lambda,\gamma}$ is the optimal $(\lambda, \gamma)$-admissible quadratic polynomial while $\alpha_2(\lambda,\gamma)$ is  the corresponding optimal gap.
\end{proof}

With the optimal  $(\lambda,\gamma)$-admissible degree 2 polynomial identified, we now look back towards its application of approximate PCP, and test the effectiveness of this polynomial against ridge regression.

\section{Ridge Regression and $(\lambda,\gamma)$-Admissible Polynomials in PCP}\label{sec:RR-Poly}
This section is devoted to answering the following question: To achieve a $(\gamma, \epsilon)$-approximate PCP for a vector at threshold $\lambda$ with least computational complexity, under what conditions should the ridge regression function be replaced by the optimal $(\lambda,\gamma)$-admissible polynomial of degree $1$ or $2$?

To answer this question, let us recall the computations are involved in computing $(\gamma, \epsilon)$-approximate PCP. The projection matrix $P_{(A,\lambda)}$ associated with $A$ at the threshold $\lambda$ can be equivalently expressed as follows:
$$
P_{(A,\lambda)} = \frac{1}{2}(I+\mathrm{sgn}(2r_\lambda(A^\top A)-I)) =  \frac{1}{2}(I+\mathrm{sgn}(p_1^{\lambda,\gamma}(A^\top A)))=\frac{1}{2}(I+\mathrm{sgn}(p_2^{\lambda,\gamma}(A^\top A))),
$$
where $r_\lambda(x)=\frac{x}{x+\lambda}$ is associated with ridge regression, $p_1^{\lambda,\gamma}(x)=\frac{x-\lambda}{1-\lambda}$ given by \eqref{eq:optimal-n=1} is the optimal degree 1 $(\lambda,\gamma)$-admissible polynomial, and $p_2^{\lambda,\gamma}(x)$ given by Theorem~\ref{Thm:Optimal2} is the optimal $(\lambda,\gamma)$-admissible quadratic polynomial. The above three different formulations for $P_{(A,\lambda)}$ naturally yield three different approaches to compute a $(\gamma, \epsilon)$-approximate PCP for any vector by replacing the signum function with its Chebyshev approximation, given in \eqref{eq:chenyshev-sign}. The minimal degree $n$ of the Chebyshev approximation provided in Lemma~\ref{lemma:Allen-Zhu-Chebyshev} is
$$
n = \frac{1}{\sqrt{2}\alpha} \log\left(\frac{3}{\epsilon \alpha^2}\right),
$$
where the values of $\alpha$  are determined by the functions $r_\lambda$,  $p_1^{\lambda,\gamma}$, and $p_2^{\lambda,\gamma}$.  More precisely, the values of $\alpha$ for determining the degrees corresponding to $r_\lambda$,  $p_1^{\lambda,\gamma}$, and $p_2^{\lambda,\gamma}$ respectively are
$\alpha_r(\lambda,\gamma)=\frac{\gamma}{2+\gamma}$, $\alpha_1(\lambda, \gamma)=\frac{\lambda\gamma}{1-\lambda}$, and $\alpha_2(\lambda, \gamma)$ given in Theorem~\ref{Thm:Optimal2}.

Once the signum function is replaced by a Chebyshev approximation, the overall computational complexity depends not only on the degree of the Chebyshev approximation, but also on the complexity of multiplying $r_\lambda(A^\top A)$, $p_1^{\lambda,\gamma}(A^\top A)$, or  $p_2^{\lambda,\gamma}(A^\top A)$ with a vector. Given any vector $y$, the computational cost of $p_1^{\lambda,\gamma}(A^\top A)y$ is dominated by that of $A^\top A y$, i.e., a single matrix-vector multiplication; the computational cost of $p_2^{\lambda,\gamma}(A^\top A)y$ is determined by two matrix-vector multiplications of this form. Unlike  $p_1^{\lambda,\gamma}(A^\top A)y$ and $p_2^{\lambda,\gamma}(A^\top A)y$ which can be evaluated exactly, $r_\lambda(A^\top A)y=(A^\top A+\lambda I)^{-1}A^\top Ay$ can be only approximated without inverting a matrix. Theoretically, $(A^\top A+\lambda I)^{-1}$ can be written as a series of $A^\top A$, so it is reasonable to assume that the computational cost of $r_\lambda(A^\top A)y$ is at least that of two matrix-vector multiplications.

Based on the discussion above, we propose a strategy of choosing ridge regression, the optimal $(\lambda,\gamma)$-admissible polynomial of degree $1$ or $2$ for a $(\gamma, \epsilon)$-approximate PCP for a vector at threshold $\lambda$. For given parameters $\lambda$ and $\gamma$, we compute the value
\begin{equation*}\label{eqn:maxalpha}
\max\{\alpha_r(\lambda, \gamma),2\alpha_1(\lambda,\gamma), \alpha_2(\lambda,\gamma)\}.
\end{equation*}
If this value equals $\alpha_r(\lambda, \gamma)$, we choose the ridge regression function $r_\lambda$; if this value equals $2\alpha_1(\lambda,\gamma)$, we choose $p_1^{\lambda,\gamma}$; and if this value equals $\alpha_2(\lambda,\gamma)$, we choose $p_2^{\lambda,\gamma}$.
With this strategy, we will show that there exist two numbers, say $b_1$ and $b_2$ with $b_1<b_2$, in the interval $(0,\frac{1}{2})$, such that we should use the ridge regression function if $\lambda\in (0,b_1)$; the optimal $(\lambda,\gamma)$-admissible polynomial of degree $2$ if $\lambda \in [b_1,b_2)$; or the optimal $(\lambda,\gamma)$-admissible polynomial of degree $1$ if $\lambda \in [b_2, \frac{1}{2})$. To find these two numbers, we conduct a study on a comparison of $2\alpha_1$ and $\alpha_2$ and a comparison of $\alpha_r$ and $\alpha_2$.

A comparison of $2\alpha_1$ and $\alpha_2$ is established in the following result.
\begin{proposition}\label{prop:deg2-deg1}
Let $\lambda \in (0,\frac{1}{2})$ and $\gamma \in (0,1)$. Then $\alpha_{2}(\lambda,\gamma)\geq 2\alpha_1(\lambda,\gamma)$ if $\lambda\in \left(0,\frac{2+2\sqrt{2}-\gamma}{8+6\sqrt{2}-\gamma}\right]$ and $\alpha_{2}(\lambda,\gamma)\le 2\alpha_1(\lambda,\gamma)$ if $\lambda\in \left [\frac{2+2\sqrt{2}-\gamma}{8+6\sqrt{2}-\gamma}, \frac{1}{2}\right]$.
\end{proposition}
\begin{proof} \ \ We first notice that $\frac{1}{3+2\sqrt{2} -\gamma} < \frac{2+2\sqrt{2}-\gamma}{8+6\sqrt{2}-\gamma} < 1-\frac{1}{2}\sqrt{2}$. As shown in Theorem~\ref{Thm:Optimal2}, the explicit expression of $\alpha_2(\lambda,\gamma)$ depends on the value of $\lambda$, therefore, the proof will be given by checking the ratio
$$
R(\lambda):=\alpha_2(\lambda,\gamma)/\alpha_1(\lambda,\gamma)
$$
for $\lambda$ in the intervals $(0,\frac{1}{3+2\sqrt{2}-\gamma}]$,  $[\frac{1}{3+2\sqrt{2}-\gamma}, 1-\frac{\sqrt{2}}{2}]$, and $[1-\frac{\sqrt{2}}{2}, \frac{1}{2})$.  Note that $\alpha_1(\alpha,\gamma)=\frac{\lambda \gamma}{1-\lambda}$ from \eqref{eq:optimal-n=1}. To view this ratio, we plot this ratio over the interval $(0,1/2)$ for three different values of $\gamma$ in Figure~\ref{fig:RatioD2D1}. This informally confirms the conclusion of the result. The formal proof follows.
\begin{figure}[htpb]
 \centering
 \begin{tabular}{ccc}
\includegraphics[scale=0.40]{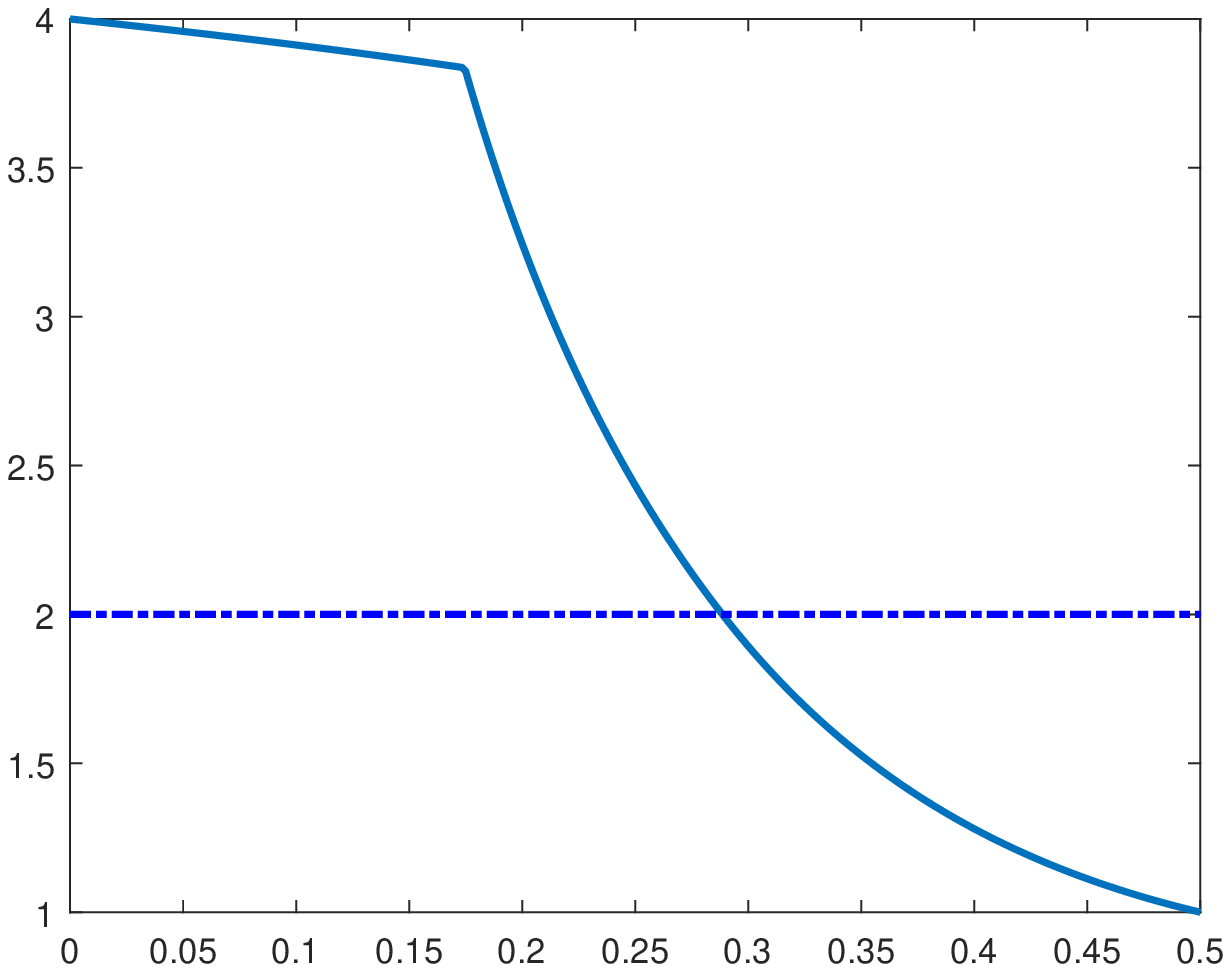}&
\includegraphics[scale=0.40]{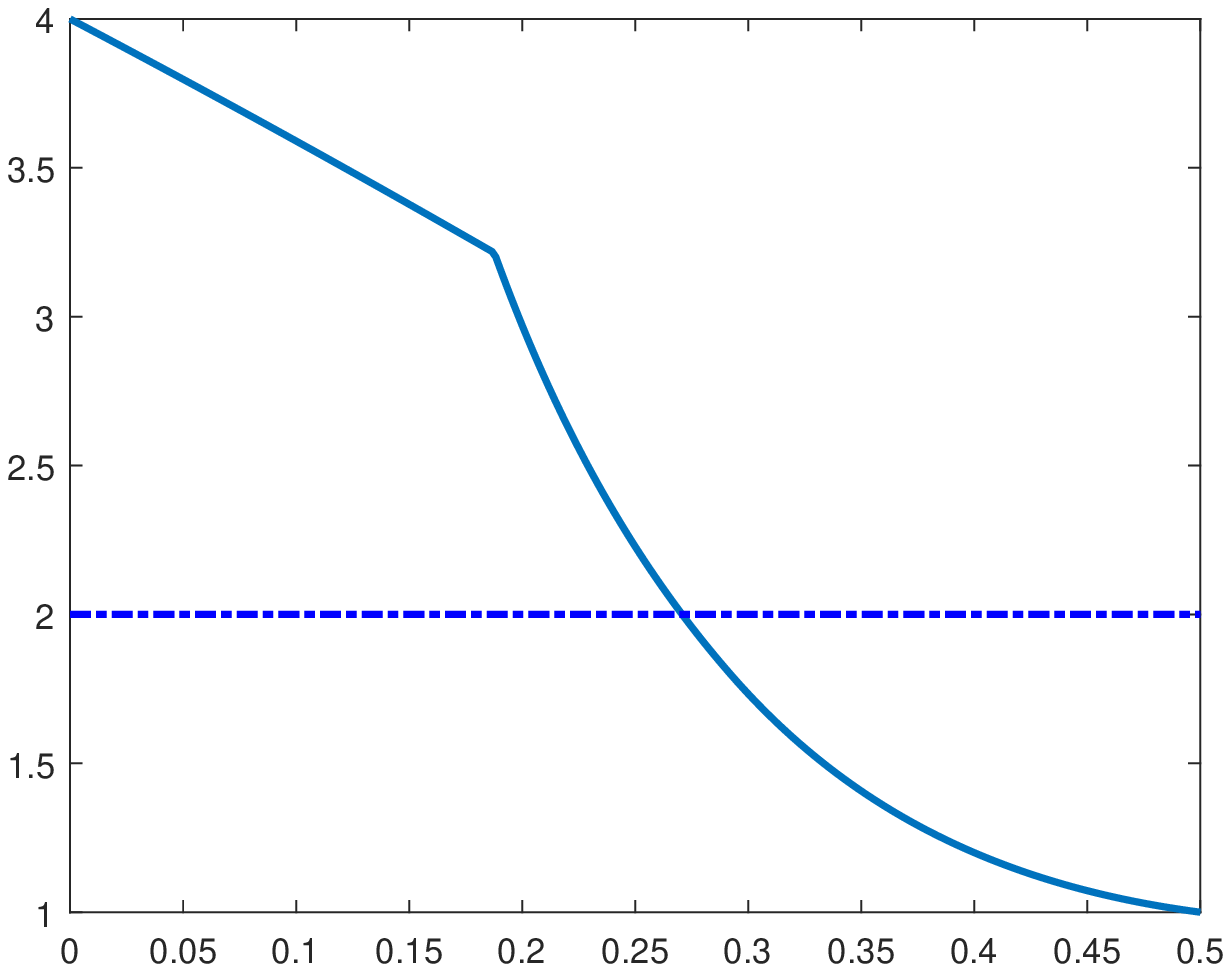}&
\includegraphics[scale=0.40]{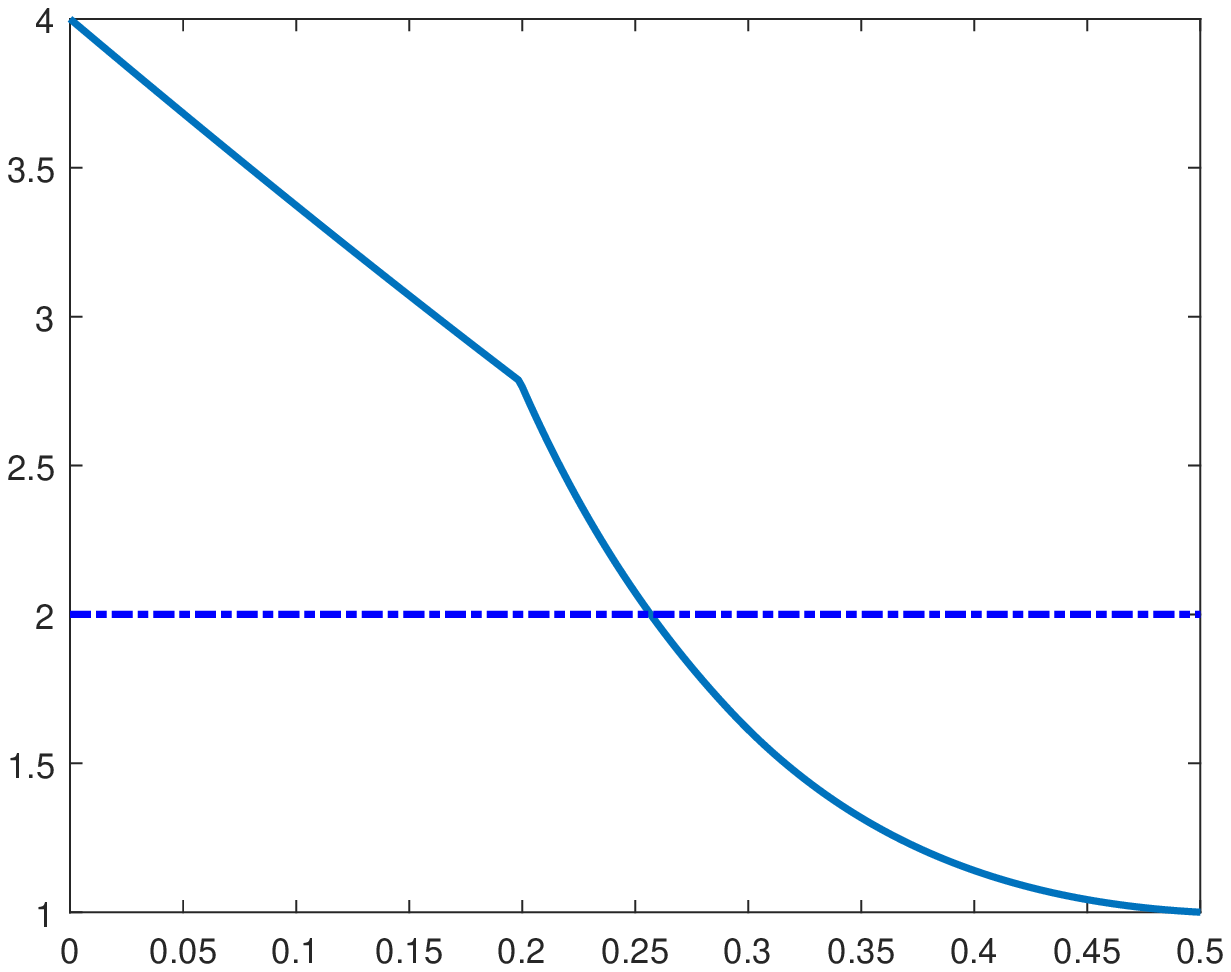}\\
(a) $\gamma=0.1$ &(b) $\gamma=0.5$ &(c) $\gamma=0.8$
\end{tabular}
\caption{The plots (solid lines) of the ratio $R(\lambda):=\alpha_2(\lambda,\gamma)/\alpha_1(\lambda,\gamma)$ over the interval $(0,1/2)$.}
\label{fig:RatioD2D1}
\end{figure}

(i) $\lambda \in (0,\frac{1}{3+2\sqrt{2}-\gamma}]$: By Theorem~\ref{Thm:Optimal2}, $\alpha_2(\lambda,\gamma)=\frac{4\gamma\lambda(1-\lambda)}{(1+\lambda\gamma-\lambda)^2}$. Then,
$R(\lambda)=\frac{4(1-\lambda)^2}{(1+\lambda\gamma-\lambda)^2}$. Since $0<1+\lambda\gamma-\lambda<1$ and $(1-\lambda)^2$ is decreasing on $(0,\frac{1}{3+2\sqrt{2}-\gamma}]$, we have
$$
R(\lambda)>4(1-\lambda)^2 \ge 4 \left(\frac{2+2\sqrt{2}-\gamma}{3+2\sqrt{2}-\gamma}\right)^2,
$$
which is greater than $2$ for every $\lambda \in (0,\frac{1}{3+2\sqrt{2}-\gamma}]$ and $\gamma \in (0,1)$.

(ii) $\lambda \in [\frac{1}{3+2\sqrt{2}-\gamma}, 1-\frac{\sqrt{2}}{2}]$:  Note that $\alpha_2(\lambda,\gamma)=\frac{\gamma(2+2\sqrt{2}-\gamma)}{ (3+2\sqrt{2})}$, independent of $\lambda$. We have $R(\lambda)=\frac{(2+2\sqrt{2}-\gamma)(1-\lambda)}{ (3+2\sqrt{2})\lambda}$, a strictly decreasing function with respect to $\lambda$. We can check that $R(\lambda) \ge 2$ if $\lambda \le \frac{2+2\sqrt{2}-\gamma}{8+6\sqrt{2}-\gamma}$ and $R(\lambda) \le 2$ if $\lambda \ge \frac{2+2\sqrt{2}-\gamma}{8+6\sqrt{2}-\gamma}$.

(iii) $\lambda \in [1-\frac{\sqrt{2}}{2}, \frac{1}{2})$: We have $\alpha_2(\lambda,\gamma)=\frac{\lambda(2\lambda-1)\gamma(1+\gamma)}{1-\lambda}+\gamma$ by  Theorem~\ref{Thm:Optimal2}. Then, after some simplifications we have
$$
R(\lambda)=\frac{1}{\lambda}-(2+\gamma)+(2+2\gamma)\lambda.
$$
We want to claim $R(\lambda)<2$ for $\lambda \in [\frac{1}{3+2\sqrt{2}-\gamma}, 1-\frac{\sqrt{2}}{2}]$, it suffices to show that $f(\lambda):=(2+2\lambda)\lambda^2-(4+\gamma)\lambda+1<0$ on $[\frac{1}{3+2\sqrt{2}-\gamma}, 1-\frac{\sqrt{2}}{2}]$. Since $\frac{(4+\gamma)+\sqrt{\gamma^2+8}}{4+4\gamma)}$, the largest root of the polynomial $f$, is bigger than $\frac{1}{2}$, and $f(1-\frac{1}{2}\sqrt{2})=-4\sqrt{2}+2\gamma-\frac{11}{2}\sqrt{2}\gamma<0$, we conclude that $f(\lambda)$ indeed is negative on $[\frac{1}{3+2\sqrt{2}-\gamma}, 1-\frac{\sqrt{2}}{2}]$.

Combining the above results, the conclusion of this proposition follows.
\end{proof}

Next, we compare the values of $\alpha_2(\lambda,\gamma)$ and $\alpha_r(\lambda,\gamma)$ for $\lambda \in(0,  \frac{1}{2})$ and $\gamma \in (0,1)$.

\begin{proposition}\label{prop:deg2-ridge}
Let $\lambda \in (0,\frac{1}{2})$ and $\gamma \in (0,1)$. Then $\alpha_{2}(\lambda,\gamma)\geq \alpha_r(\lambda,\gamma)$ if $\lambda  \in \left[ \frac{ 5+\gamma-2\sqrt{4+2\gamma} }{(9+2\gamma +\gamma^2)}, \frac{1}{2}\right )$ and $\alpha_{2}(\lambda,\gamma)\le \alpha_r(\lambda,\gamma)$ if $\lambda  \in \left(0,  \frac{ 5+\gamma-2\sqrt{4+2\gamma} }{(9+2\gamma +\gamma^2)}\right )$.
\end{proposition}

\begin{proof}\ \
We first notice that $0<\frac{ 5+\gamma-2\sqrt{4+2\gamma} }{(9+2\gamma +\gamma^2)} < \frac{1}{3+2\sqrt{2} -\gamma}$ for all $\gamma \in (0,1)$. As shown in Theorem~\ref{Thm:Optimal2}, the explicit expression of $\alpha_2(\lambda,\gamma)$ depends on the value of $\lambda$, therefore, the proof will be given by checking the ratio
$$
R(\lambda):=\alpha_2(\alpha,\gamma)/\alpha_r(\alpha,\gamma)
$$
for $\lambda$ in the intervals $(0,\frac{1}{3+2\sqrt{2}-\gamma}]$,  $[\frac{1}{3+2\sqrt{2}-\gamma}, 1-\frac{\sqrt{2}}{2}]$, and $[1-\frac{\sqrt{2}}{2}, \frac{1}{2})$. To view this ratio, we plot this ratio over the interval $(0,1/2)$ for three different values of $\gamma$ in Figure~\ref{fig:RatioD2R}. This informally confirms the conclusion of the result. The formal proof follows.
\begin{figure}[htpb]
 \centering
 \begin{tabular}{ccc}
\includegraphics[scale=0.40]{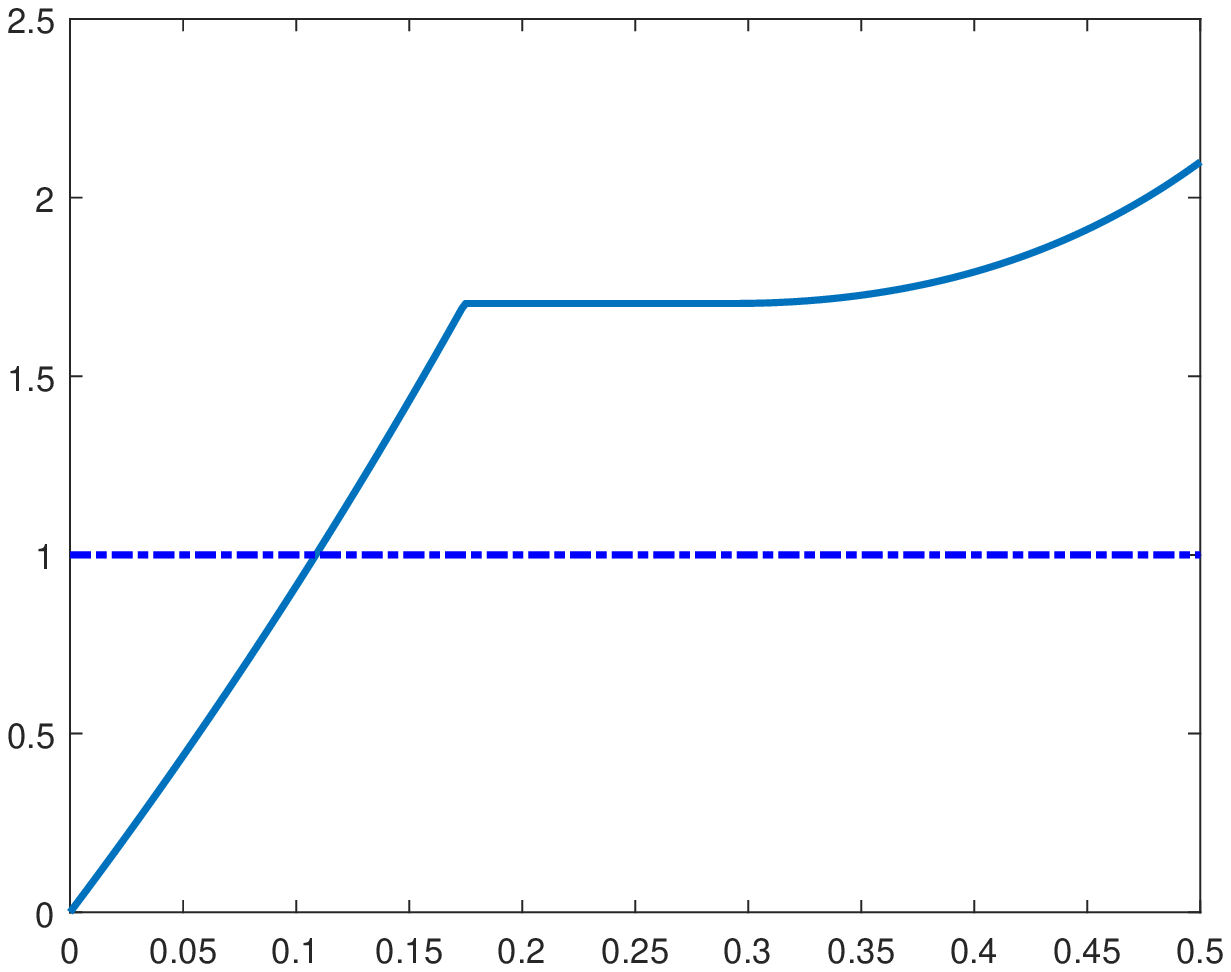}&
\includegraphics[scale=0.40]{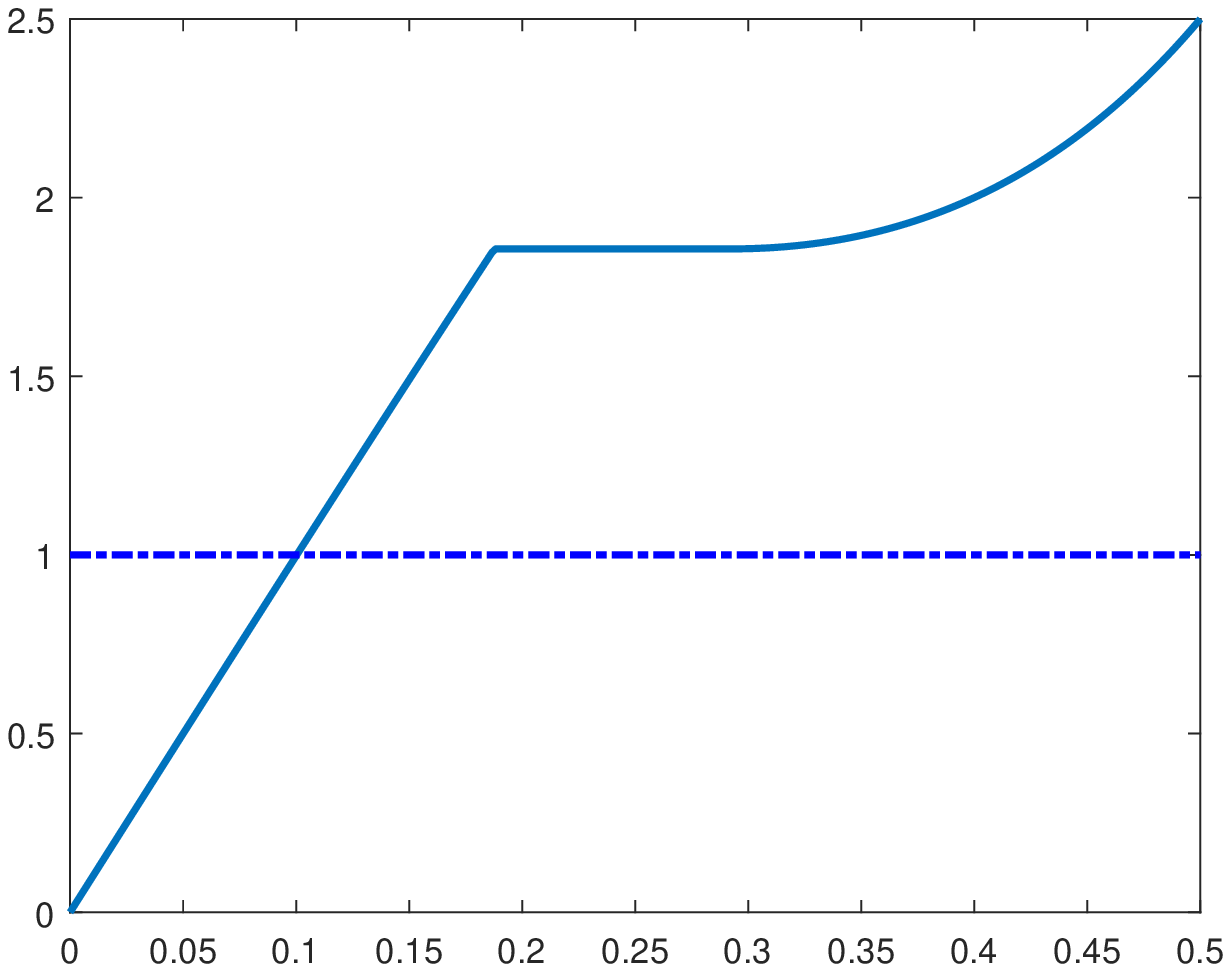}&
\includegraphics[scale=0.40]{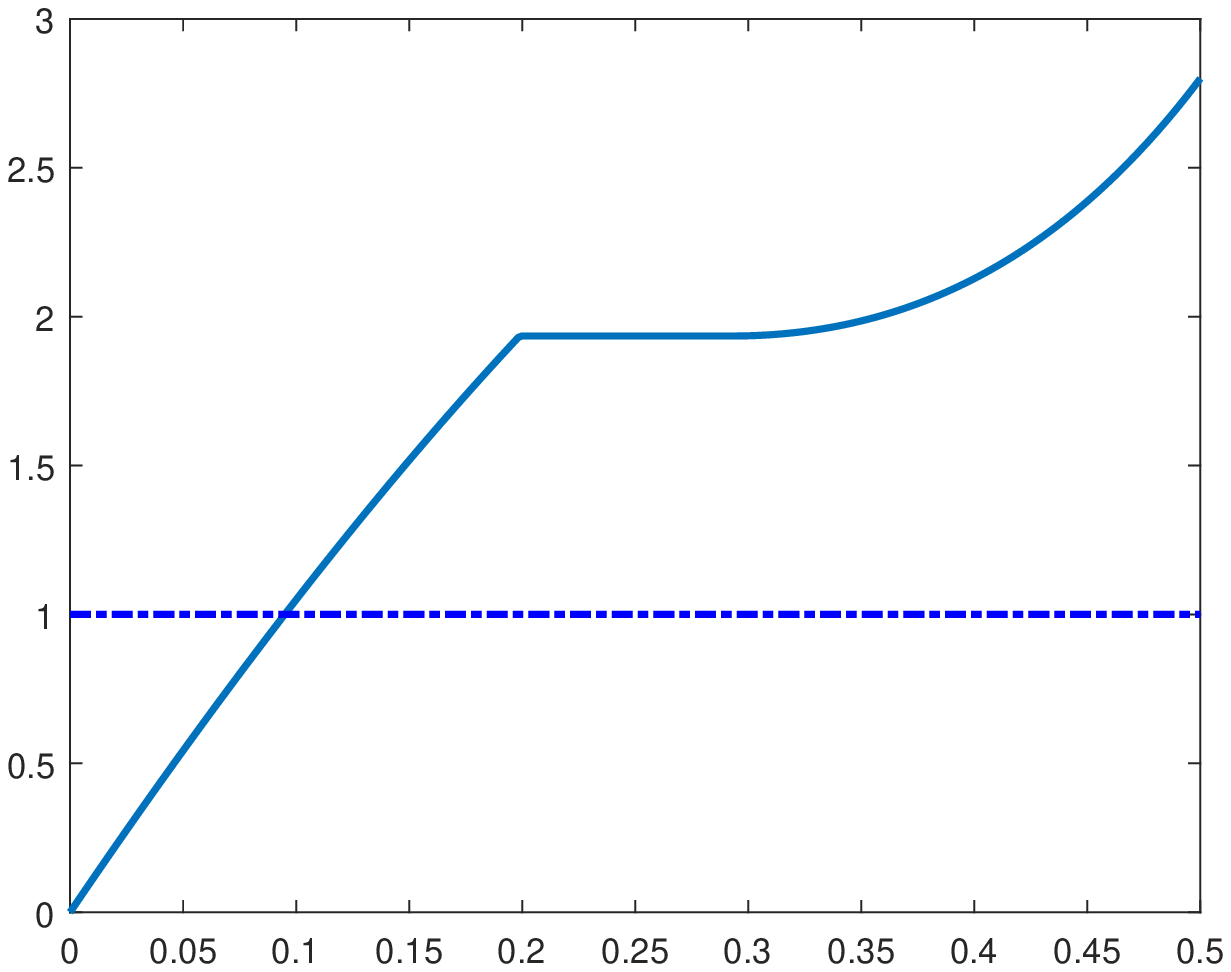}\\
(a) $\gamma=0.1$ &(b) $\gamma=0.5$ &(c) $\gamma=0.8$
\end{tabular}
\caption{The plots (solid lines) of the ratio $R(\lambda):=\alpha_2(\lambda,\gamma)/\alpha_r(\lambda,\gamma)$ over the interval $(0,1/2)$.}
\label{fig:RatioD2R}
\end{figure}

(i) $\lambda \in (0,\frac{1}{3+2\sqrt{2}-\gamma}]$: By Theorem~\ref{Thm:Optimal2}, $\alpha_2(\lambda,\gamma)=\frac{4\gamma\lambda(1-\lambda)}{(1+\lambda\gamma-\lambda)^2}$. Then, $R(\lambda)=\frac{4\lambda(1-\lambda)(2+\gamma)}{(1+\lambda\gamma-\lambda)^2}$. Hence, $R(\lambda) \ge 1$ if and only if $f(\lambda):=4\lambda(1-\lambda)(2+\gamma)-(1+\lambda\gamma-\lambda)^2\le 0$. The function $f(\lambda)=-(9+2\gamma+\gamma^2)\lambda^2+(10+2\gamma)\lambda-1$ is concave and has two roots $\lambda_1=\frac{ 5+\gamma-2\sqrt{4+2\gamma} }{(9+2\gamma +\gamma^2)}$ and $\lambda_2=\frac{ 5+\gamma+2\sqrt{4+2\gamma} }{(9+2\gamma +\gamma^2)}$. Notice that $\lambda_2>\frac{1}{2}$ and $f\left(\frac{1}{3+2\sqrt{2}-\gamma}\right)>0$, we know that $R(\lambda) \ge 1$ if $\lambda \in (\lambda_1, \frac{1}{3+2\sqrt{2}-\gamma}]$ and $R(\lambda) \le 1$ if $\lambda \in (0, \lambda_1]$.

(ii) $\lambda \in [\frac{1}{3+2\sqrt{2}-\gamma}, 1-\frac{\sqrt{2}}{2}]$:  Note that $\alpha_2(\lambda,\gamma)=\frac{\gamma(2+2\sqrt{2}-\gamma)}{ (3+2\sqrt{2})}$, independent of $\lambda$. We have $R(\lambda)=\frac{4+4\sqrt{2}+2\sqrt{2}\gamma -\gamma^2}{3+2\sqrt{2}}$ which is clearly bigger than $1$ for all $\gamma \in (0, 1)$.

(iii) $\lambda \in [1-\frac{\sqrt{2}}{2}, \frac{1}{2})$: We have $\alpha_2(\lambda,\gamma)=\frac{\lambda(2\lambda-1)\gamma(1+\gamma)}{1-\lambda}+\gamma$ by  Theorem~\ref{Thm:Optimal2}. Then $R(\lambda)=\frac{((2\lambda-1)(\lambda +\lambda\gamma)+1-\lambda)(2+\gamma)}{(1-\lambda)}$. Actually, we can show $R(\lambda)>1$ for all $\lambda \in (0, 1)$. In this situation, it is equivalent to showing $f(\lambda):=((2\lambda-1)(\lambda +\lambda\gamma)+1-\lambda)(2+\gamma)-(1-\lambda) >0$. Since after simplifications $f(\lambda)=(1+\gamma)((4+2\gamma)\lambda^2-(3+\gamma)\lambda+1)$ and $(3+\gamma)^2-4(4+2\gamma)<0$, we know that $f(\lambda)>0$ for all $\lambda$.

The result of this proposition follows immediately from the above discussions.
\end{proof}

By using Propositions~\ref{prop:deg2-deg1} and \ref{prop:deg2-ridge}, we now set
\begin{equation}\label{def:b}
b_1=\frac{ (5+\gamma)-2\sqrt{4+2\gamma} }{9+2\gamma +\gamma^2} \quad \mbox{and} \quad
b_2=\frac{2+2\sqrt{2}-\gamma}{8+6\sqrt{2}-\gamma}.
\end{equation}
Both numbers are in $(0, \frac{1}{2})$ and $b_1<b_2$. From these propositions, our answer to the question raised at the beginning of this section is summarized as follows. According to the value of $\gamma$, the interval $(0, \frac{1}{2})$ is partitioned into three subintervals $(0, b_1)$, $[b_1, b_2)$, and $[ b_2, \frac{1}{2})$. Then, according to the value of $\lambda$, we will choose $r_\lambda$, $p_2^{\lambda,\gamma}$, or  $p_1^{\lambda,\gamma}$ in computing a $(\lambda,\gamma)$-approximate PCP if $\lambda$ is in $(0, b_1)$, $[b_1, b_2)$, or $\left[ b_2, \frac{1}{2}\right]$, respectively.

To end this section, we present a $(\lambda,\gamma)$-approximate PCP algorithm using the optimal $(\lambda,\gamma)$-admissible polynomials. See  \texttt{PolyPCP} in Algorithm~\ref{alg:FSS}. In particular, we rename \texttt{PolyPCP} as  \texttt{Poly1PCP} if $p^{\lambda,\gamma}_1$ is used and rename \texttt{PolyPCP} as  \texttt{Poly2PCP} if $p^{\lambda,\gamma}_2$ is adopted.

\begin{algorithm} \label{alg:FSS}
\SetKwInOut{Input}{Input}
\SetKwInOut{Output}{Output}
\caption{\texttt{PolyPCP}$(A,\chi,\lambda,\gamma,n)$}
	
	\Input{ $A \in \mathbb{R}^{m \times d}, \text{ } \|A\| =1$; \quad $ \chi \in \mathbb{R}^d$;\quad  $\lambda > 0, $ threshold; \newline
	 \quad $\gamma \in (0,1),$ PCP approximation ratio \quad $n$, degree of Chebyshev approximation}
	
	\Output{vector that approximates $\zeta \approx P_{(A,\lambda)}(\chi)$}
	
	$\gamma \leftarrow \max\{\gamma, \frac{\log(n)}{n}\}$
		
	$\kappa \leftarrow 2 \alpha^2$  \hskip 1cm $\diamond$ Here  $\alpha=\alpha_1(\lambda,\gamma)$ (or $\alpha=\alpha_2(\lambda,\gamma)$) if $p^{\lambda,\gamma}_1$ (or $p^{\lambda,\gamma}_2$) is used;
		
	Define $c_k = \frac{2-\mathbbm{1}[k=0]}{n+1} \sum_{j=0}^n \sqrt{2} \cos\left(\frac{k(j+.5)\pi}{n+1}\right)\left(1+\kappa - \cos\left(\frac{(j+.5)\pi}{n+1}\right)\right)^{-1/2} $\
	
	$b_{n+1} \leftarrow  0$, $b_n  \leftarrow c_n \chi$\
	
	\For{$ r = n-1$ \textbf{to} $0$}{
	
	$w\leftarrow (1+\kappa)b_{r+1}- 2 p^{\lambda,\gamma}(A^\top A)  p^{\lambda,\gamma}(A^\top A) b_{r+1}$  \hskip 1cm $\diamond$ Here  $p^{\lambda,\gamma}$ is $p^{\lambda,\gamma}_1$ (or $p^{\lambda,\gamma}_2$) \
	
	$b_r \leftarrow  2w-b_{r+2}+c_r \chi$\
	
}
$u \leftarrow p^{\lambda,\gamma}(A^\top A) (b_0-w)$\

\Return{$\frac{1}{2} (u+\chi)$}
	
\end{algorithm}

\section{Experiments}\label{sec:experiments}
In this section, we provide numerical results to illustrate the effectiveness of the proposed algorithm \texttt{PolyPCP} for $(\lambda,\gamma)$-approximate PCP and to validate the theoretical results.

We begin with the methodology for creating the datasets for our experiments. We set the parameters $\lambda \in (0, \frac{1}{2}]$ and $\gamma \in (0,1)$. The size of the matrix $A$ in our simulation is $2000 \times 2000$. We generate this matrix in two different ways. In the first way, we begin with a $2000 \times 2000$ random matrix $B$ whose SVD is $B=U \widehat{\Sigma}V^\top$. Here $U$ and $V$ are two orthogonal matrices and the diagonal entries of $\widehat{\Sigma}$ are the singular values of $B$. We divide each diagonal entry of $\widehat{\Sigma}$ by $\sigma_1$ (the largest signal value of $B$), and then the diagonal entries in the interval $(\sqrt{(1-\gamma)\lambda}, \sqrt{(1+\gamma)\lambda})$ are changed to zeros. We denote the resulting matrix $\Sigma$. By using $U$, $V$ and $\Sigma$, we define $A=U {\Sigma}V^\top$.  We refer to this matrix $A$ as the \textit{random eigenvalue distribution} matrix.
In the second way, we generate a $2000 \times 2000$ diagonal matrix $\Sigma$ by uniformly choosing half of diagonal entries from $[0, \sqrt{(1-\gamma)\lambda}]$ and another half from $[\sqrt{(1+\gamma)\lambda}, 1]$. We then construct $A=U\Sigma V^\top$ where $U$ and $V$ are two $2000 \times 2000$ random orthogonal matrices. This is similar to the synthetic data used in \cite{Allen-Zhu-Li:ICML:17}.  We refer to this matrix as the \textit{uniform eigenvalue distribution} matrix.

We measure the accuracy of the $(\lambda,\gamma)$-approximate PCP from \texttt{PolyPCP} and \texttt{QuickPCP} by the relative projection error defined as follows.
Let $P_{(A,\lambda)} \chi$ be the projection of $\chi$ onto the span of the singular vectors of $A$ whose corresponding singular values are at least $\sqrt{\lambda}$, and  let $\xi$ be the output from either \texttt{PolyPCP} or \texttt{QuickPCP}. The relative  projection error is defined as
\begin{equation}
E={\|\xi-P_{(A,\lambda)}  \chi \|_2}/{\|P_{(A,\lambda)}  \chi\|_2}.
\end{equation}

We fix $\gamma=0.1$ in our experiments, and get $b_1\approx 0.1087$ and $b_2 \approx 0.2886$ by \eqref{def:b}. As a direct result of Proposition \ref{prop:deg2-ridge}, we know that $\alpha_r(\lambda,\gamma) > \alpha_2(\lambda,\gamma)$ whenever $\lambda \in (0,b_1)$, and the negation of this inequality holds when $\lambda \in [b_1,\frac{1}{2})$. We also know that $\alpha_1(\lambda,\gamma) < \alpha_2(\lambda,\gamma)$ for any $\lambda \in (0,\frac{1}{2})$. Therefore, for a fixed degree Chebyshev polynomial, $g_n$, we can expect the relative projection error of \texttt{QuickPCP} to be lower than that of \texttt{Poly1PCP} and \texttt{Poly2PCP} if $\lambda\in (0,b_1)$, and the relative projection error of  \texttt{Poly2PCP} to be lower than that of  \texttt{QuickPCP} and \texttt{Poly1PCP} if $\lambda \in (b_1,\frac{1}{2})$. When considering the time for each algorithm to converge within a fixed error tolerance, as a result of Proposition \ref{prop:deg2-deg1} we can expect that \texttt{Poly2PCP} will converge the fastest when $\lambda \in (b_1,b_2)$ and \texttt{Poly1PCP} will converge the fastest when $\lambda \in (b_2,\frac{1}{2})$. For the case when $\lambda \in (0,b_1)$, the variability in the cost of ridge regression prevents us from knowing if \texttt{Poly2PCP} or \texttt{QuickPCP} will converge in the least amount of time.

In the first experiment, we test the uniform eigenvalue distribution matrices. We choose four different $\lambda$ values $0.05, 0.15, 0.3,$ and $0.48$. Since $0.05 \in (0, b_1)$ and based on the discussion given at the end of Section~\ref{sec:RR-Poly}, we expect that \texttt{QuickPCP}, for this $\lambda$ value, performs better than both \texttt{Poly1PCP} and \texttt{Poly2PCP} in terms of the relative  projection error. Similarly, since $0.15, 0.3$ and $.48$ are in the interval $[b_1, \frac{1}{2})$, we expect that \texttt{Poly2PCP} has a lower relative error than both \texttt{Poly1PCP} and \texttt{QuickPCP}. To confirm these expectations, we compute the relative projection  errors of the results from \texttt{QuickPCP}, \texttt{Poly1PCP}, and \texttt{Poly2PCP}, letting the degree of the Chebyshev polynomial vary from 1 to 200. We remark that for \texttt{QuickPCP}, MATLAB's conjugate gradient method with an error tolerance of $\epsilon'=10^{-13}$ is adopted as a ridge regression solver.  Each graph in Figure~\ref{fig:3} plots the relative error in logarithmic scale on the vertical axis and the degree of the Chebyshev approximation to the signum function on the horizontal axis. The three curves in each graph are for \texttt{QuickPCP}, \texttt{Poly1PCP}, and \texttt{Poly2PCP}. The relative projection  error plotted is averaged over 100 experiments. The results in all graphs of Figure~\ref{fig:3} are consistent with our theoretical results in Section~\ref{sec:RR-Poly}. We also see from Figure~\ref{fig:3}(d) that \texttt{Poly1PCP}, and \texttt{Poly2PCP} perform similarly when $\lambda=0.48$. This is mainly due to the ratio $\alpha_2(\lambda,\gamma)/\alpha_1(\lambda,\gamma)$ tending to 1 as $\lambda$ approaches 0.5 (also see Figure~\ref{fig:RatioD2D1}(a)).

\begin{figure}
 \centering
 \begin{tabular}{cc}
\includegraphics[scale=.25]{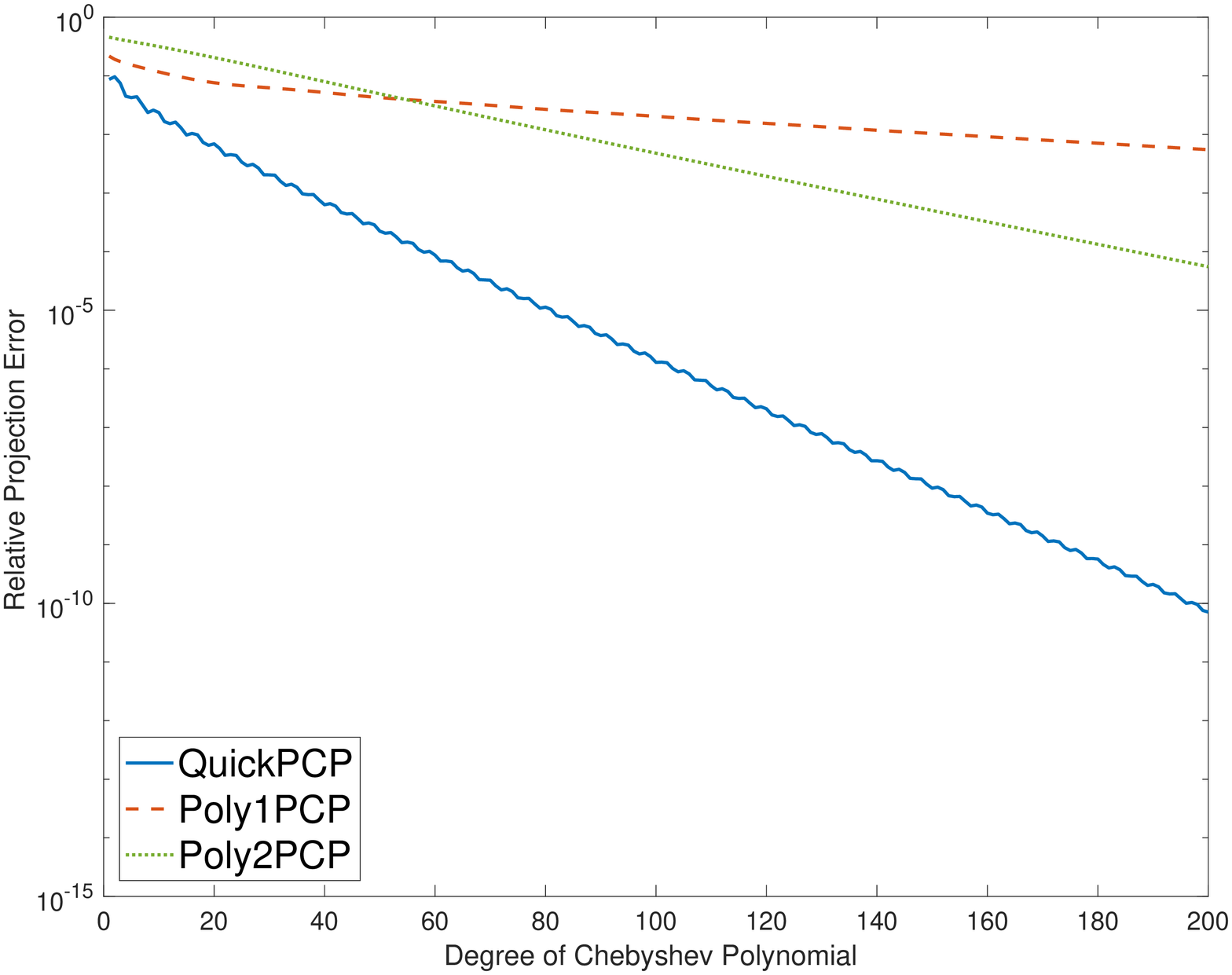}&
\includegraphics[scale=.25]{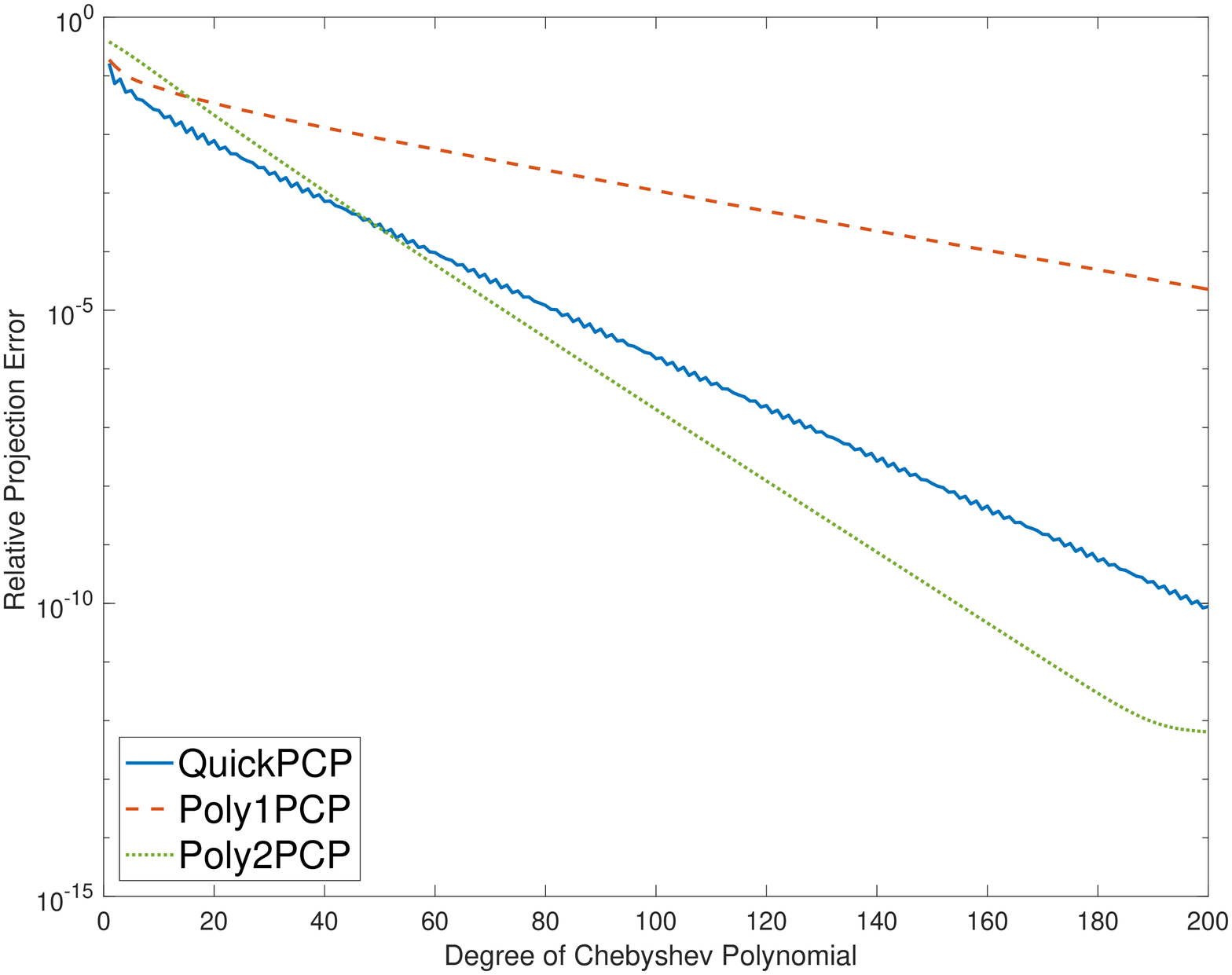}\\
(a) $\lambda=0.05$ &(b) $\lambda=0.15$\\
\includegraphics[scale=.25]{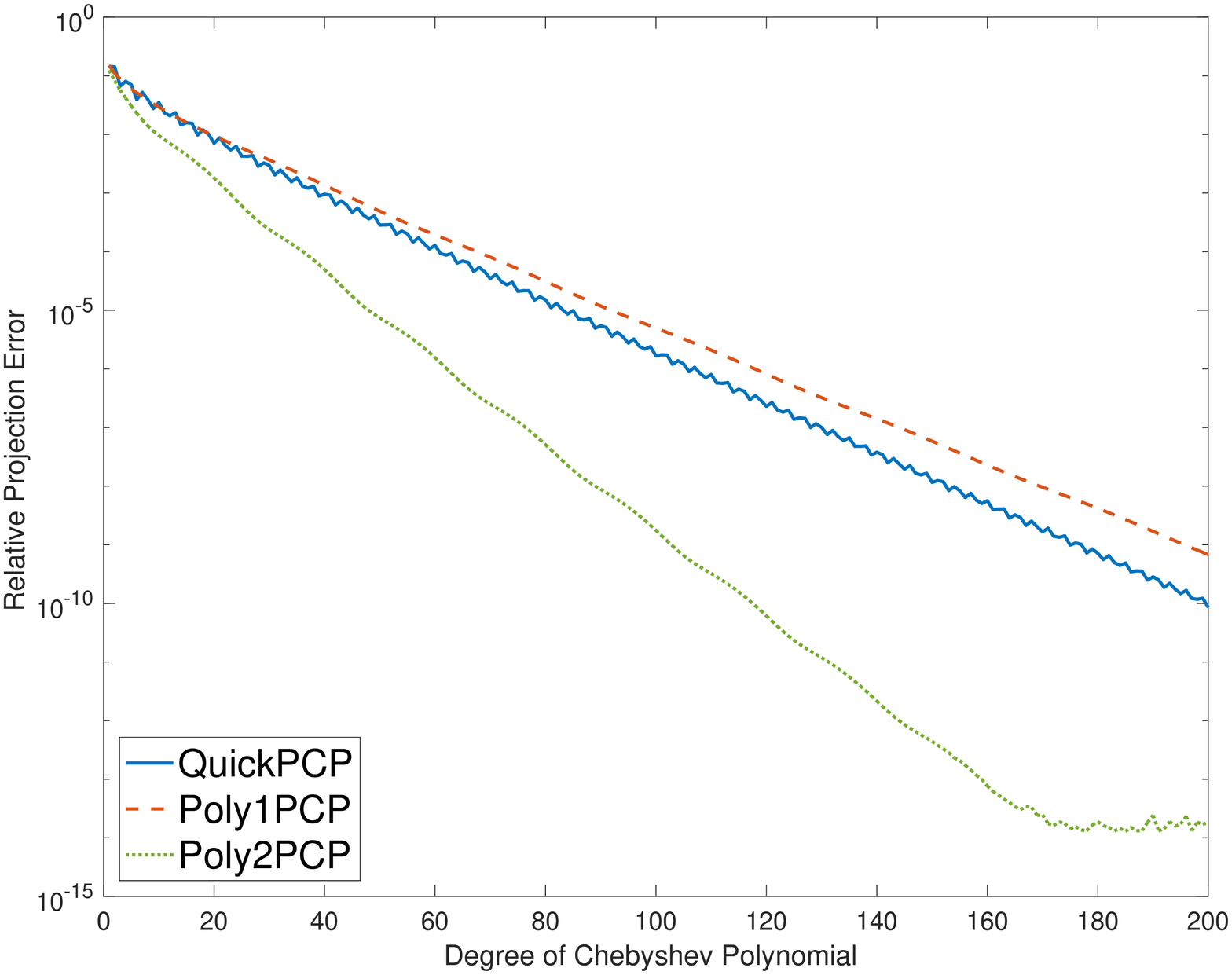} &
\includegraphics[scale=.25]{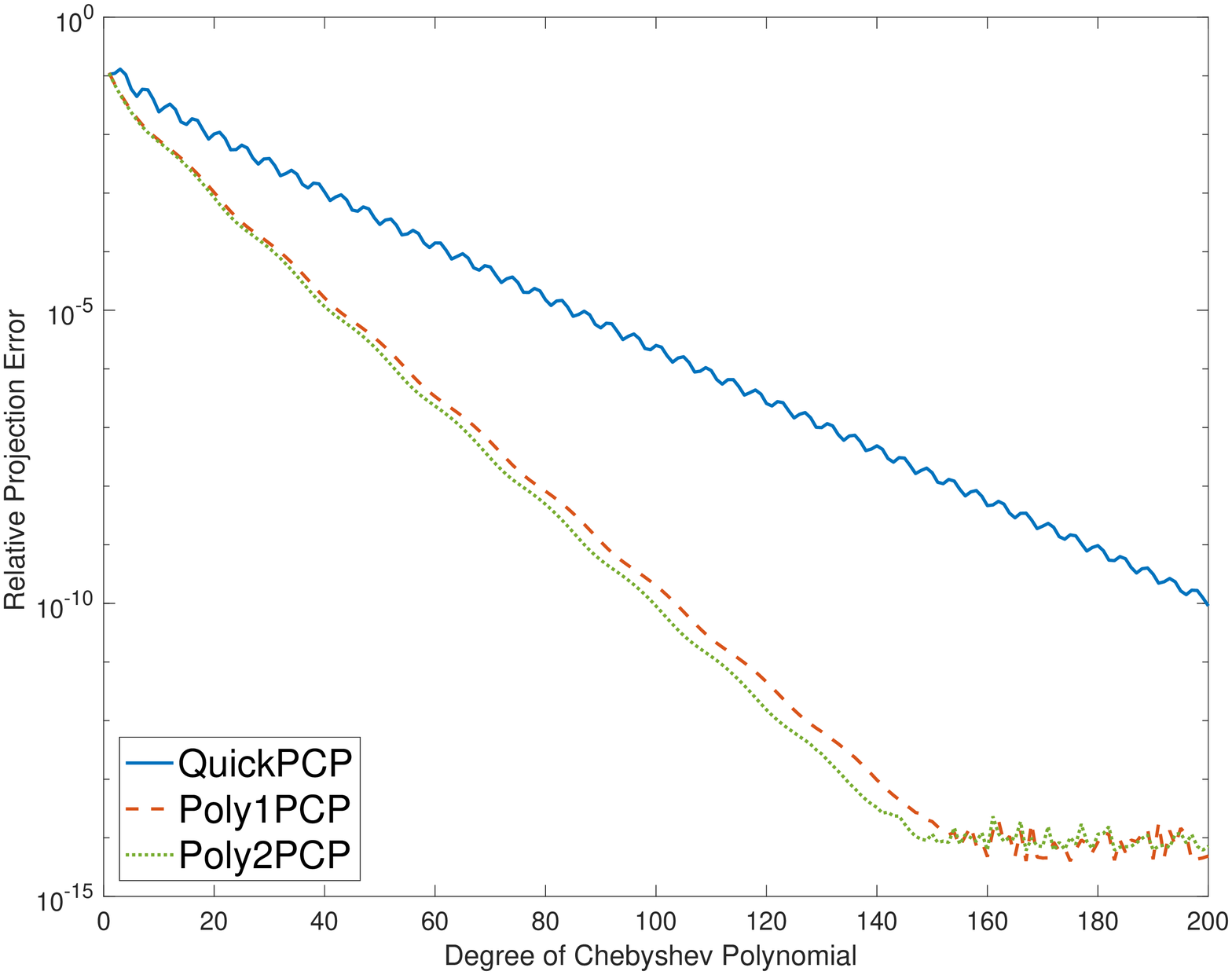} \\
 (c) $\lambda=0.3$ &(d)$\lambda=0.48$
 \end{tabular}
\caption{The relative  projection errors obtained through \texttt{QuickPCP}, \texttt{Poly1PCP}, and \texttt{Poly2PCP}, plotted against the degree of the Chebyshev approximation to the signum function.}
\label{fig:3}
\end{figure}

In our second experiment, we examine the performance of \texttt{QuickPCP}, \texttt{Poly1PCP}, and \texttt{Poly2PCP} in terms of CPU time when the relative errors of their outputs are less than $10^{-12}$. The data for this experiment was generated and timed using MATLAB 2017a. The computer that generated these times is a 2017 MacBook Pro with a 3.1 GHz Intel Core i5 Processor with 16 GB of memory. In Figure~\ref{fig5}(a), when using a randomly generated eigenvalue distribution we can see that when $\lambda \in (.05,.5)$, both \texttt{Poly1PCP}, and \texttt{Poly2PCP} outperform \texttt{QuickPCP}, with a time reduction over a factor of 10 when $\lambda \ge .175$. In this case, we see that the convergence time for \texttt{QuickPCP} does not significantly decrease as $\lambda$ increases, while \texttt{Poly1PCP} and \texttt{Poly2PCP} do. In Figure~\ref{fig5}(b), the matrix $A$ has a uniform eigenvalue distribution. Again, both \texttt{Poly1PCP}, and \texttt{Poly2PCP} outperform \texttt{QuickPCP}. We can see a convergence time reduction of close to a factor of 50 when $\lambda$ approaches $.5$. As our theoretical results predicted, when $\lambda \approx b_2$, the algorithms \texttt{Poly1PCP} and \texttt{Poly2PCP} converge in roughly the same amount of time for both the uniform and random eigenvalue distributions, with the predicted algorithm performing better when $\lambda$ is well above and below $b_2$.


\begin{figure} \centering
 \begin{tabular}{cc}
\includegraphics[scale=.27]{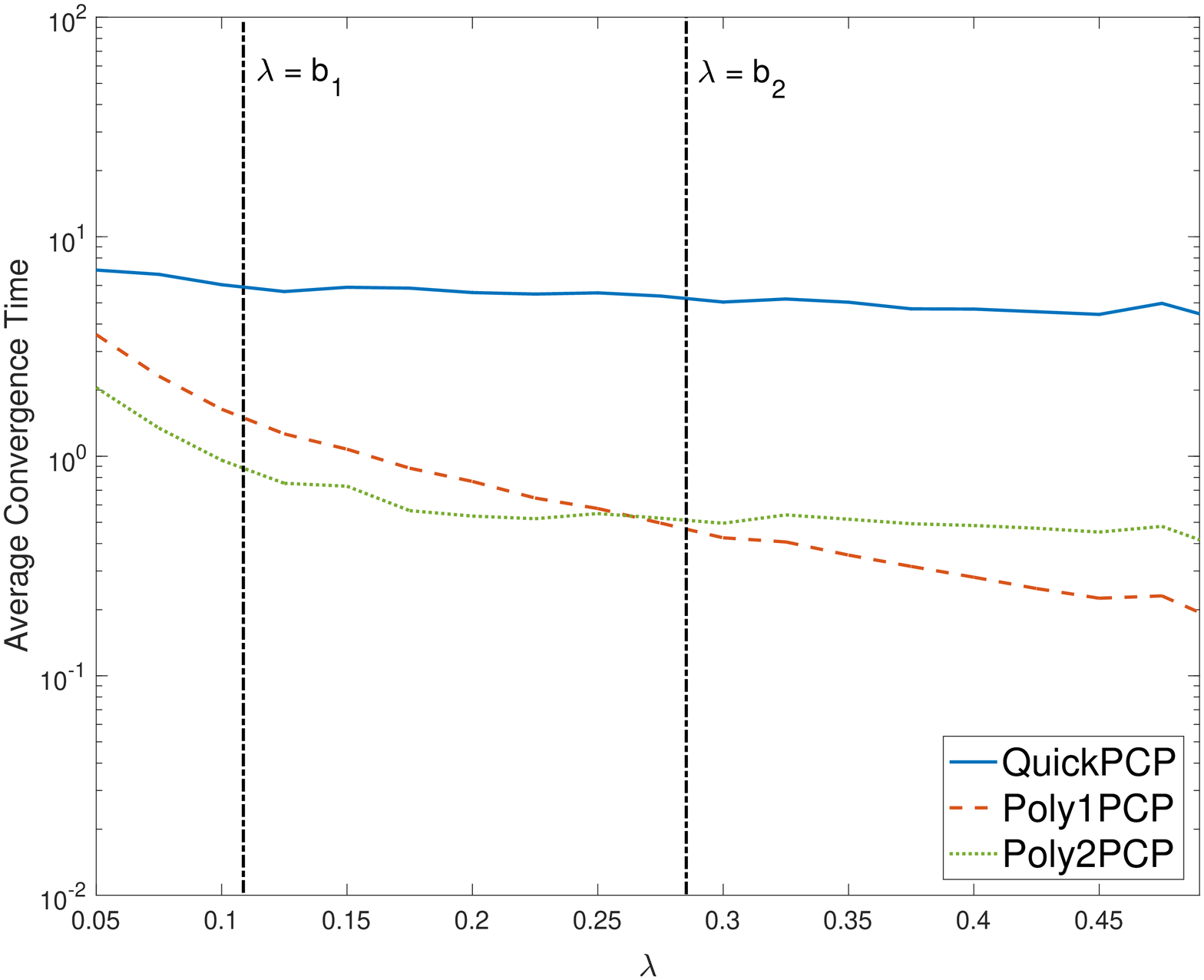} &
\includegraphics[scale=.27]{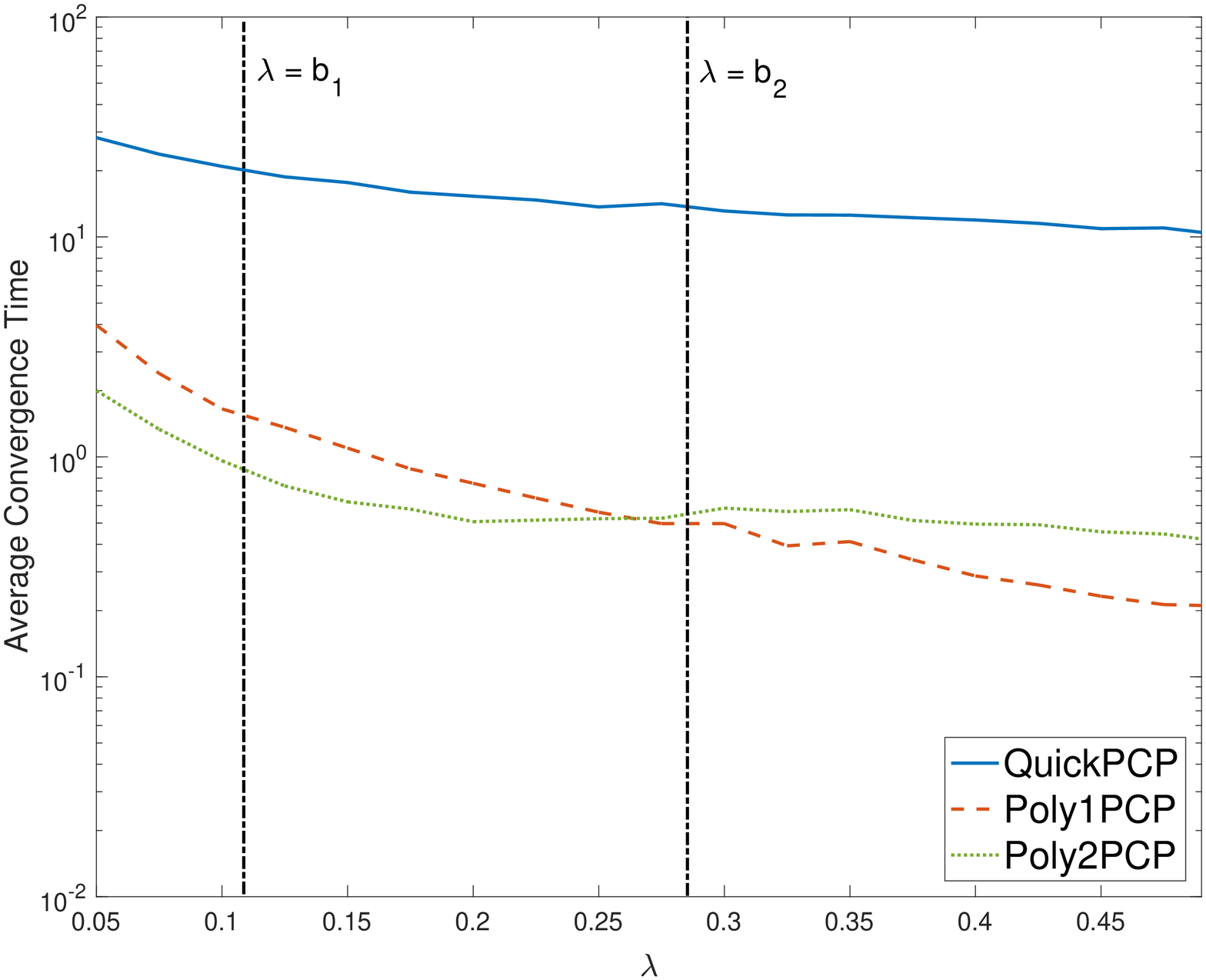}\\
(a) Random Eigenvalue Distribution & (b) Uniform Eigenvalue Distribution
\end{tabular}
\caption{The total CPU time in $\log$ seconds for  \texttt{QuickPCP}, \texttt{Poly1PCP}, and \texttt{Poly2PCP}.}
\label{fig5}
\end{figure}

\section{Conclusion}\label{sec:conclusion}
In studying approximate PCP algorithms, we have found that the replacement of a black box ridge regression subroutine in \cite{Allen-Zhu-Li:ICML:17,Frostig-Musco-Musco-Sidford:ICML:16} with optimal low degree polynomials can lead to an improvement of convergence time by well over an order of magnitude. We have shown theoretically that for a large proportion of $\lambda$ values, degree one and two polynomials reduce the required degree of the Chebyshev approximation to the signum function to guarantee approximate PCP when compared to ridge regression. Experimentally, we have confirmed these findings, and showed that these polynomials decrease convergence time on an even larger interval of $\lambda$ values. In some instances, we have shown that the overall reduction in time for convergence between our approach and the previous state of the art algorithm can even surpass a factor of 40. In our work, we have also formulated a continuous constraint optimization problem that can be solved for higher order polynomials, and these results may provide even more significant improvements on a larger range of $\lambda$ values.

\section*{Acknowledgement of Support and Disclaimer}
SDF and LS acknowledge the Government's support in the publication of this paper. This material is based upon work funded by AFRL under Contract No. FA8750-16-3-6003. Any opinions, findings, and conclusions or recommendations expressed in this material are those of the authors and do not necessarily reflect the views of AFRL.

Lixin Shen is partially supported by the US National Science Foundation under grant DMS-1522332.

\end{document}